\documentclass[11pt]{article}
\usepackage{amsmath, graphicx, amsfonts,amssymb}
\usepackage{amsfonts,mathrsfs, color, amsthm}
\usepackage{bm}

\usepackage[bookmarks=true,
bookmarksnumbered=true, breaklinks=true,
pdfstartview=FitH, hyperfigures=false,
plainpages=false, naturalnames=true,
colorlinks=true,pagebackref=false,
pdfpagelabels]{hyperref}

\usepackage{enumitem}
\usepackage{hyperref}
\hypersetup{
  colorlinks   = true,  
  urlcolor     = blue,  
  linkcolor    = blue, 
  citecolor   = red    }
  
\newtheorem{theorem}{Theorem}[section]
\newtheorem{definition}[theorem]{Definition}
\newtheorem{lemma}[theorem]{Lemma}
\newtheorem{corollary}[theorem]{Corollary}

\newtheorem{proposition}[theorem]{Proposition}

\newtheorem{problem}[theorem]{Problem}

\numberwithin{equation}{section}

 \def\R{\mathbb{R}}
 \def\K{\mathscr K}
 \def\C{\mathcal C}
 \def\dom{{\mathrm {dom}}}
 \def\L{\mathscr L}
 \def\A{\mathscr A}
 \def\sphere{S^{n-1}}
 
 \def\ball{B_1}

\begin{document}
\title{Geometry of log-concave functions: the $L_p$ Asplund sum and the $L_{p}$ Minkowski problem \footnote{Keywords:   Log-concave functions, $L_p$ Asplund sum,   $L_p$ Minkowski problem, $L_p$ surface area measure,  Moment measures, Pr\'ekopa-Leindler inequality,  Variational formula.}}

\author{Niufa Fang, Sudan Xing  and Deping Ye}
\date{}
\maketitle

\begin{abstract}
		
The aim of this paper is to develop a basic framework of the $L_p$ theory for the geometry of log-concave functions, which can be viewed as a functional ``lifting" of  the $L_p$ Brunn-Minkowski theory for convex bodies. To fulfill this goal, by  combining the $L_p$ Asplund sum of log-concave functions for all $p>1$ and the total mass, we obtain a Pr\'ekopa-Leindler type inequality and propose a definition for the first variation of the total mass in the $L_p$ setting. Based on these, we further establish an $L_p$ Minkowski type inequality related to the first variation of the total mass and derive a variational formula which motivates the definition of our $L_p$ surface area measure for log-concave functions. Consequently,  the $L_p$ Minkowski problem for log-concave functions, which aims to characterize the $L_p$ surface area measure for log-concave functions, is introduced. The existence of solutions to  the $L_p$ Minkowski problem for log-concave functions is obtained for $p>1$ under some mild conditions on the pre-given Borel measures.		
		
\vskip 0.3cm

2010 Mathematics Subject Classification: 26B25 (primary), 26D10, 52A40. 

\end{abstract}
		
\section{Introduction and overview of the main results} \setcounter{equation}{0}
 A function $f: \R^n\to \R$ is said to be log-concave if $-\log f: \R^n\to \R\cup\{+\infty\}$ is a convex function. It is well-known that the  log-concave functions behave in many aspects akin to convex bodies (i.e., compact convex sets with nonempty interiors, and the set of all convex bodies is denoted by  $\K^n$), which is considered as an analytic ``lifting" of geometry of convex bodies. In convex geometry, the Brunn-Minkowski theory (and its extensions) for convex bodies encompasses a large and growing range of fundamental results on the algebraic and geometric properties of convex bodies, and hence to study the parallel algebraic and geometric properties of log-concave functions is of great significance and in great demand. Recent years have witnessed that many results and notions in the Brunn-Minkowski theory (and its extensions) for convex bodies have found their functional analogues,  including but not limited to the functional Blaschke-Santal\'{o} type inequality and its inverse \cite{AKM04,AS15,Ball88-2,BBF14,FM07,FM08, KM05},  (John,  Lutwak-Yang-Zhang, and L\"owner) ellipsoids  for log-concave functions \cite{AMJV18, FZ18,LSW19jga}, Rogers-Shephard type inequality and its reverse  for log-concave functions \cite{Alo19, AAMJV19, AMJV16,Col06}, the affine surface areas 
  for log-concave functions \cite{AKSW12,CFGLSW16,   CW14, CW15, CY16,LSW19}, the variation and Minkowski type problems related to log-concave  functions \cite{CF13,CK15, Kla14, San16},  and (isoperimetric) inequalities  related to log-concave functions \cite{ABM20, ABM, AFS20, Bob99,  BL76,    Lin17,  Rot13}. Other contributions include e.g., \cite{  AM09,AS19,  MR13} among others. 
 
The celebrated  Pr\'ekopa-Leindler inequality  \eqref{PLineq} for log-concave functions has a formula similar to the  famous dimensional-free Brunn-Minkowski inequality for convex bodies  (see \eqref{p-BM-2-2} for $p=1$).   Note that 
the Pr\'ekopa-Leindler inequality indeed works for more general functions, see \cite{Gar02, Lei72,Pre71,Pre73,Pre75}. 
The key ingredients in the Pr\'ekopa-Leindler inequality for log-concave functions are the total mass of a  log-concave function $f$ given by \begin{equation}\label{totalmass}
J(f)=\int_{\R^n} f(x) \, dx, 
\end{equation} where $\,dx$ denotes the Lebesgue measure in $\R^n$ and the Asplund sum $f\oplus g$ of two log-concave functions $f$ and $g$   defined by \begin{equation}\label{A-sum-6-23} f\oplus g(x)=\sup_{x=x_1+x_2} f(x_1)g(x_2), \end{equation}   which is again a log-concave function. Likewise, the volume and the Minkowski addition for convex bodies (which is equivalent to \eqref{Lp-addition} for $p=1$ for convex bodies) are the key ingredients in the Brunn-Minkowski inequality, and their combination naturally results in an elegant variational formula (i.e., \eqref{variation-convex bodies}  for $p=1$). Such a variational formula defines a crucial concept in convex geometry: the surface area measure $S_K$ for convex body $K\in \K ^n$. Notably,  the Minkowski problem \cite{Minkowski1987, Minkowski1903},
   aiming to characterize the surface area measure for convex bodies, can be solved by using the Euler-Lagrange equation based on the variational formula (i.e., \eqref{variation-convex bodies}  for $p=1$). 
Lifting to the functional setting,  the first variation of the total mass of log-concave functions with respect to the Asplund sum has been established by Colesanti and Fragal\`a  in their groundbreaking work  \cite{CF13}, and from there one can see that  things become way more complicated for log-concave functions.  The challenge partially comes from  the facts that the total mass is defined on $\R^n$ instead of compact sets ($\sphere$ for convex bodies) and the family of log-concave functions contains too many ``unfavourable" log-concave functions (for example, those log-concave functions which are not smooth enough and/or whose domains are not nice enough). 
 
The variational formula of   the total mass   for log-concave functions with respect to the Asplund sum produces a measure for log-concave functions which is similar to $S_K$ for $K\in \K^n$; such a measure was named as  the surface area measure for a log-concave function  in \cite{CF13} (which is also known as the moment measure of convex functions in \cite{CK15}).  That is, for a log-concave function $f=e^{-\psi}$ 
with $0 <J(f) < \infty$, the moment measure of $\psi$ or the surface area measure of $f$, denoted by $\mu(f, \cdot)$, was defined as the push-forward measure of $e^{-\psi}\,dx$ under $\nabla\psi$ where   $\nabla \psi$ denotes the gradient of $\psi$.  That is,  \begin{equation*}
\int_{\R^n} g(y) \, d\mu(f, y) = \int_{\R^n} g(\nabla \psi(x)) \, e^{-\psi(x)}\, dx
\end{equation*} for every Borel function $g$ such that
$g \in L^1(\mu(f, \cdot) )$ or $g$ is non-negative.   In their groundbreaking work  \cite{CK15}, Cordero-Erausquin and Klartag studied the Minkowski type problem for log-concave functions in order to characterize the moment measure.
Note that the Minkowski problem for log-concave functions in its full formulation was independently posed by Colesanti and Fragal\`a  in  \cite{CF13}, and finding solutions to this problem (in its full formulation) seems to be very important but highly intractable.   Under certain conditions on log-concave functions (in particular, the essential continuity for the negative of their logarithms), Cordero-Erausquin and Klartag  were able to solve the Minkowski problem for log-concave functions  (see \cite[Proposition 1]{CK15} for necessity and  \cite[Theorem 2]{CK15} for sufficiency and uniqueness). Santambrogio in \cite{San16} presented a solution to the Minkowski problem for log-concave functions  by using the technique of the optimal mass transportation.

 Although many concepts for log-concave functions have been studied, what brings into our attention is the missing of the ``$L_p$ addition" of log-concave functions and the corresponding variational formula. These shall be analogous to the $L_p$ addition \cite{Firey} and the related variational formula for convex bodies \cite{Lut93} which are the crucial elements in the rapidly developing $L_p$ Brunn-Minkowski theory for convex bodies. In the same manner as the Brunn-Minkowski theory for convex bodies, the $L_p$ theory for $p>1$ was developed based on the brilliant variational formula \eqref{variation-convex bodies} by Lutwak in his influential work \cite{Lut93}. Again, the variational formula \eqref{variation-convex bodies} relies on the combination of the $L_p$ addition for convex bodies  
 and the volume, which can be extended to all $0\neq p\in \R$ (see e.g., \cite{ZHY2016}). Naturally, the $L_p$ surface area measure  $h_K^{1-p} S_{K}$ for a convex body $K\in \K^n_{(o)}$  (see Section \ref{Preliminaries} for notations) can be defined from \eqref{variation-convex bodies}, and consequently the $L_p$ Minkowski problem \cite{Lut93} characterizing the $L_p$ surface area measure for convex bodies 
 can be posed (see Problem \ref{lp-min-6-26} for its precise statement). The $L_p$ Minkowski problem is a milestone   in convex geometry and receives enormous attention in many areas of mathematics.   Contributions to the $L_p$ Minkowski problem include   \cite{BLYZ13, chen,CW06,HLW15,HLYZ05,JLW15,JLZ16, LW13, Lut93, LO95, LYZ04,   Uma03,Zhu15,Zhu15jfa,Zhu17} among others.  
 
Our main goals  in this paper are to establish  a basic framework of the functional ``lifting" of the $L_p$ Brunn-Minkowski theory for convex bodies, and thence extend the scheme for log-concave functions to  an $L_p$ setting. Based on the $L_p$ Asplund sum of log-concave functions which naturally generalizes the Asplund sum, for $p>1$, we establish a variational formula related to the total mass, define the $L_p$ surface area measure for log-concave functions,  and  study the $L_p$ Minkowski problem which characterizes the $L_p$ surface area measure for log-concave functions. It is our hope that these contributions provide some useful tools in the development of  geometry of log-concave functions.
 
More specifically, Section \ref{p-sum of log-concave function} dedicates to study the properties for the $L_p$ Asplund sum of log-concave functions.  Let $f=e^{-\varphi}\in \A_0$ and $g=e^{-\psi}\in \A_0$  (see Sections \ref{Preliminaries} and \ref{p-sum of log-concave function} for notations). For $p>1$ and $\alpha, \beta>0$,  the $L_p$ Asplund sum  $\alpha\cdot_p f\,\oplus_p\,\beta\cdot_p g$ may be formulated  by 
		\begin{equation} \label{p-addition-6-23}  \alpha\cdot_p f\,\oplus_p\,\beta\cdot_p g:=e^{-\phi^*} \ \ \ \mathrm{with} \ \ \  \phi:=\big[\alpha(\varphi^*)^p+\beta(\psi^*)^p\big]^{\frac{1}{p}},\end{equation}  where $\varphi^*$ denotes the Fenchel conjugate of $\varphi$  given in \eqref{def-dual-1}. This definition is meaningful as it can be seen in Proposition \ref{close-Lp-A-sum} that $\alpha\cdot_p f\,\oplus_p\,\beta\cdot_p g\in \A_0$ for $p> 1$. We establish the following Pr\'ekopa-Leindler type  inequality in Theorem \ref{prekopainequality}:  {\em for $f, g\in\A_0$, $\lambda\in (0,1)$ and $p>1$,  it holds that $$
		J\big((1-\lambda)\cdot_p f\oplus_{p} \lambda\cdot_p g\big)\geq J(f)^{1-\lambda}J(g)^{\lambda}  $$
		with  equality if and only if $f=g$ on $\R^n$.} We also prove the log-concavity of $J(f\oplus _p t\cdot_p g)$ on $t\in (0, \infty)$ and $J\big((1-t)\cdot_p f\oplus_p t\cdot_p g\big)$ on $t\in (0, 1)$ in Corollary \ref{concave-PL-1}. 

In Section  \ref{sectoion-4-6-23},  we define $\delta J_{p}(f, g)$ for $p>1$,  the first variation of the total mass at $f\in \A_0$ along $g\in \A_0$ with respect to the $L_{p}$  Asplund sum, by \[\delta J_{p}(f,g):=\underset{t\rightarrow0^+}{\lim}\frac{J(f\oplus_{p} t\cdot_p g)-J(f)}{t},\] whenever the  limit exists. To find an explicit formula to describe $\delta J_{p}(f,g)$ for all $f, g\in \A_0$ seems to be intractable and even impractical. However, when $f=g\in \A_0$  such that $J(f)>0$, we are able to prove in Lemma \ref{specialcaseoffg} that, for $p>1,$   $$
\delta J_{p}(f,f)= \frac{n}{p} \int_{\R^n} f \, dx   + \frac{1}{p} \int_{\R^n} f \log f \, dx,
$$  which involves the entropy of $f$ (see \eqref{entropu-p>1} for more details).  Our main result in this section is a Minkowski type inequality for $\delta J_{p}(f, g)$ established in  Theorem \ref{minkowskiinequality}, namely, {\em 
		for  $f \in \A_0$ and $g \in\A_0$ such that $J(f)>0$ and $J(g)>0$,  one has,  for $p> 1,$
		$$
		\delta J_{p}(f,g)\geq \delta J_{p}(f,f)+J(f)\log\Big(\frac{J(g)}{J(f)}\Big)
		$$
		with equality if and only if  $f=g$.} This Minkowski type inequality is applied to obtain a unique determination of log-concave functions in Corollary \ref{cormink-uniq}.    	
		
Although, in general,  an explicit formula for $\delta J_{p}(f,g)$ is not easy to get,  under certain conditions on $f$ and $g$ (which differs from the one for $f=g$ in Section \ref{sectoion-4-6-23}), such a formula can be obtained. This result consists of our main contribution in Section \ref{section-varition-516} and  is proved in Theorem \ref{variationformula}. Roughly speaking, it asserts that  {\em   for $p>1$,  \begin{align}\label{p-derivative-6-23} 
		\delta J_{p}(f,g) =\frac{1}{p}\int_{\mathbb{R}^n}(\psi^*(y))^{p}(\varphi^*(y))^{1-p}\,d\mu(f, y) \end{align} 
		 holds for $f=e^{-\varphi}$ and $g=e^{-\psi}$ smooth enough such that $(\varphi^*)^p-c (\psi^*)^p$ is a convex function for some constant $c>0$, and  the function  $\frac{1}{\alpha} (\varphi^*)^{\alpha}$  (understood as $\log \varphi^*$ when $\alpha=0$) is convex  for some $\alpha\in [0, 1)$.}  Please see more precise statements in Theorem  \ref{variationformula} and Corollary \ref{verify-condition-1}. The proof of Theorem  \ref{variationformula} is very technical and involves a lot of rather complicated analysis on the (first and second order) differentiability  and the limit of the partial derivatives of the function $-\log (f\oplus_{p} t\cdot_p g)$.  We would like to mention that,  even  in the case $p=1$ where $\phi$ defined in \eqref{p-addition-6-23}  is linear, such analysis (for the differentiability and related limits) already exhibits its complexity as one can see in \cite{CF13}. The nonlinearity of $\phi$ defined in \eqref{p-addition-6-23} for $p>1$ does bring extra difficulty and makes our analysis even more complicated. In particular, additional conditions on $\varphi$ are needed, such as, \eqref{compatible-1} or those given in  Corollary \ref{verify-condition-1}.

Our last contribution is the study of the $L_p$ Minkowski problem for log-concave functions, which will be presented in Section  \ref{functional p Minkowski problem}.  In fact, formula \eqref{p-derivative-6-23} suggests a natural way to define the $L_p$ surface area measure for the log-concave function $f=e^{-\varphi}$, denoted as  $\mu_p(f, \cdot)$ and defined on  $\Omega_{\varphi^*}=\{y\in \R^n:  0<\varphi^*(y)<+\infty\}$ (this set is always assumed to be nonempty).  Under certain conditions on a given log-concave function $f=e^{-\varphi}$,   the measure $\mu_p(f, \cdot)$ for $p\in \R$ is given by  $$  \int_{\Omega_{\varphi^*}}g(y)\, d\mu_p (f, y)=\int_{\{x\in \R^n:\   \nabla \varphi(x)\in \Omega_{\varphi^*}\}}g(\nabla \varphi(x))(\varphi^*(\nabla \varphi(x)))^{1-p} e^{-\varphi(x)} \,dx $$ for every Borel function $g$ such that $g \in L ^ 1 (\mu_p (f,\cdot))$  or $g$ is non-negative.  We pose the following  $L_p$ Minkowski problem for log-concave functions: {\em for $p\in \R$,  find the necessary and/or sufficient conditions on a  finite nonzero Borel measure $\nu$ defined on  $\mathbb{R}^n$  so that   $\nu=\tau \mu_{p}(f,\cdot)$    holds for some log-concave function $f$ and $\tau\in\R$.}  We provide a solution to this problem under the following mild conditions on $\nu$: $\nu$ is  not supported in a lower-dimensional subspace of $\R^n$, $\nu(M_{\nu}\setminus L)>0$ holds for any bounded convex set $L\subset\R^n$ where  $M_{\nu}$ is the interior of the convex hull of the support of $\nu$, and $  \int_{\R^n} |x|^p\, d\nu(x)<\infty.$ Note that these conditions for $\nu $ are all natural; see Section \ref{functional p Minkowski problem} for more details. Our solution to the $L_p$ Minkowski problem for log-concave functions is proved in Theorem \ref{solution-lp-mink-6-2}.  That is, roughly speaking, {\em if $\nu$ is an even measure satistying the above conditions,  for  $p>1$,  there exists an even log-concave function  $f=e^{-\varphi}$, such that,   \begin{eqnarray*}
 \nu =\frac{\int_{\Omega_{\varphi^*}} \,d\nu(y)}  {\int_{\Omega_{\varphi^*}} \,d \mu_{p}(f,y) } \mu_{p}(f, \cdot)\ \ \ \mathrm{on} \ \ \Omega_{\varphi^*}.
		\end{eqnarray*}  } 
				
Finally, we would like to comment that if $\nu$ admits a density function with respect to the Lebesgue measure, say $\,d\nu(y)=h(y)\,dy$, and $f=e^{-\varphi}$ with the convex function $\varphi$ smooth enough, then  finding a solution to the $L_p$ Minkowski problem for log-concave functions requires to search for a (smooth enough) convex function $\varphi$  satisfying the following Monge-Amp\`{e}re equation:  
\begin{align*} h(y) =  \tau  \varphi^*(y)^{1-p} e^{\varphi^*(y)-\langle y, \nabla\varphi^*{y}\rangle} \det(\nabla^2\varphi^*(y)),    \ \ \ \mathrm{for} \ y\in \Omega_{\varphi^*},  \end{align*}  where $\tau\in\R$ is a constant  and $\det(\nabla^2\varphi^*(y))$ denotes the determinant of the Hessian matrix of $\varphi^*$ at $y\in \R^n$.

\section{Preliminaries and Notations}\label{Preliminaries} \setcounter{equation}{0}

This section provides preliminaries and notations required for (log-concave)  functions and convex bodies. More details can be found in \cite{BV10, Roc70,Sch14}. The paper \cite{Col17} by Colesanti is also an excellent source to learn various important topics related to log-concave functions.

Let  $\mathbb{N}$ be the set of natural numbers and $n\in \mathbb{N}$ such that $n\geq 1$. In the $n$-dimensional Euclidean space $\R^n$,  $\langle x ,  y \rangle$ stands for the standard inner product of $x,y\in\R^n$ and $|x|$ for the Euclidean norm  of $x\in \R^n$. The origin of $\R^n$ is denoted by $o$.  A set $E\subset\R^n$ is said to be origin-symmetric if $-x\in E$ for any $x\in E$.  Let $B _r = \{ x \in \R ^n \, :\, |x| \leq r \}$ be  the origin-symmetric ball with radius $r$. By  $\sphere$ we mean the unit sphere, i.e., the boundary of $\ball$. We also use $V(E), \mathrm{int}(E), \overline{E}$ and $\partial E$ to denote the volume, interior, closure and boundary of $E\subset \R^n$, respectively.  By $\mathcal{H}^k$, we mean the $k$-dimensional Hausdorff measure.

A set $K\subset \R^n$ is said to be a convex body if $K$ is a compact convex set with nonempty interior. 	The family of convex bodies is denoted by $\K ^n$. With $\K_{(o)}^n$, we mean the collection of convex bodies containing $o$ in their interiors. We say $K\in \K_{(o)}^n$ is a dilation  of $L\in \K_{(o)}^n$ if there exists a $\lambda>0$ such that $K=\lambda L$.  
 For any $K \in \K^n$, the {\it support function} of $K$, $h _K: \R^n\rightarrow \R$,  which is a sublinear function on $\R^n$ and can be used to determine $K$  uniquely,   takes the following form: 
\begin{equation}\label{support-1}  
h _K (x)= \sup_{ y \in K} \langle x,  y \rangle  \ \ \ \mathrm{for }\  x \in \R ^n.
\end{equation}

Let $\alpha, \beta>0$ be two constants. For $K, L\in \K_{(o)}^n$, define $\alpha \cdot_p K+_p\beta\cdot_p L$ for $p\geq 1$ (see e.g., \cite{Firey, Lut93, Sch14})  to be the convex body whose support function is given by  \begin{equation}\label{Lp-addition} h_{\alpha \cdot_p K+_p\beta\cdot_p L}=\big(\alpha h_K^p+\beta h_L^p\big)^{\frac{1}{p}}. \end{equation} Of course, $\alpha \cdot_p K+_p\beta\cdot_p L$  in  \eqref{Lp-addition}  reduces to the Minkowski addition of convex bodies when $p=1$.  In \cite{Lut93}, Lutwak proved the following  $L_p$ Brunn-Minkowski inequality for $p>1$:  \begin{eqnarray}\label{p-BM}
V(K+_pL)^{\frac{p}{n}}\geq V(K)^{\frac{p}{n}}+V( L)^{\frac{p}{n}}
\end{eqnarray}
with equality if and only if $K$ and $L$ are the dilation of each other.  Moreover, Lutwak in \cite{Lut93} also established the following remark variational formula: for $K, L\in \K_{(o)}^n$ and $p> 1$, one has \begin{equation}\label{variation-convex bodies} 
\frac{p}{n}\cdot \frac{d}{dt}V(K+_p t\cdot _p L)\Big|_{t=0^+}=\frac{1}{n}\cdot \int_{\sphere}h_L^p(u)h_K^{1-p}(u)\, d S_{K}(u):=V_p(K,L).	
\end{equation}  Here $S_K$ is the surface area measure of $K$ on $\sphere$, that is,  $S_K(\eta)=\mathcal{H}^{n-1}(\nu_K^{-1}(\eta))$ for any Borel set $\eta\subset \sphere$,  where $\nu_K^{-1}$ is the reverse Gauss map of $K$. Note that both  \eqref{p-BM} and \eqref{variation-convex bodies}  works for $p=1$ as well, and inequality  \eqref{p-BM} becomes the classical Brunn-Minkowski inequality (in this case, equality holds if and only if $K$ and $L$ agree up to a translation and a dilation).  We refer \cite{Sch14} and references therein for more details.

   This elegant formula \eqref{variation-convex bodies}  introduces two important geometric objects:  the $L_p$ mixed volume $V_p(K, L)$  for $K, L\in \K_{(o)}^n$  and  the $L_p$ surface area measure of $K\in  \K_{(o)}^n$  given by $\, d S_p(K, \cdot)=h_K^{1-p}\, d S_{K}.$  Associated to the $L_p$ mixed volume is the following $L_p$ Minkowski inequality  for $p> 1$: 
\begin{eqnarray*}\label{p-M}
	V_{p}(K, L)\geq V(K)^{\frac{n-p}{n}}V(L)^{\frac{p}{n}},
\end{eqnarray*} 
 with equality if and only if $K$ and $L$ are the dilation of each other \cite{Lut93}.  Related to the $L_p$ surface area measure is the following well-known  $L_p$ Minkowski problem \cite{Lut93}. 
\begin{problem}[The $L_p$ Minkowski problem for convex bodies]\label{lp-min-6-26} Let $p\in \R$. What are the   necessary and sufficient conditions on a nonzero finite Borel measure $\mu$  on  $S^{n-1}$ such that  there exists a convex body  $K
	\subset \mathbb{R}^n$  satisfying $S_p(K, \cdot)=\mu$?\end{problem} 

Throughout the paper, let $\varphi : \R^n\to\R\cup\{+\infty\}$ be a convex function, that is,  $$\varphi(\lambda x+(1-\lambda)y)\leq \lambda \varphi(x)+(1-\lambda) \varphi(y)$$ holds  for any $\lambda\in (0, 1)$ and $x, y\in \mathrm{dom} (\varphi)$, where $
\mathrm{dom} (\varphi)=\{x\in\mathbb{R}^n\,:\, \varphi(x)<+\infty\}$ denotes the domain of $\varphi.$ Denote by $\C$ the set of all convex functions from $\R^n$  to $\R\cup\{+\infty\}$.  Clearly,  $\mathrm{dom} (\varphi)$ is a convex set for  any   $\varphi \in \C$. A convex function $\varphi\in \C$ is said to be  {\it proper} if $\mathrm{dom} (\varphi) \neq \emptyset$. By $\C^1(E)$ we mean the set of all functions $\varphi\in \C$ such that $\varphi$ is continuously differentiable on $E\subseteq \mathrm{int} (\dom (\varphi))$. Similarly, by $\varphi\in \C^ 2_+(E)$ we mean that $\varphi\in \C$ is twice continuously differentiable and its Hessian matrix is positive definite on $E\subseteq \mathrm{int} (\dom (\varphi))$. 

Let $\varphi: \R^n\to\R\cup\{+\infty\}$ be a function  (not necessary a convex function) such that $\varphi(x)<\infty$ for some $x\in \R^n$. The {\it Fenchel conjugate}  of $\varphi$ is defined by:
\begin{equation}
\varphi^*(y)=\sup_{x\in\R^n}\big\{ \langle x , y \rangle -\varphi(x)\big\} \ \ \ \mathrm{for \ all} \ y \in \R ^n.\label{def-dual-1}
\end{equation}  
For $\alpha>0$, \eqref{def-dual-1} yields that, for all $y\in \R^n$, \begin{equation} \label{const-mul-dual} (\alpha \varphi)^*(y)=\sup_{x\in\R^n}\big\{ \langle x , y \rangle -\alpha \varphi(x)\big\}=\alpha \sup_{x\in\R^n}\big\{ \langle x , y/\alpha \rangle -\varphi(x)\big\}=\alpha \varphi^*(y/\alpha). \end{equation}  Clearly,  $\varphi^{**}\leq \varphi$,  $\varphi^*\in \C$, and $\varphi^*$ is always lower semi-continuous. It can also be checked that if $\varphi\leq \psi$, then $\varphi^*\geq \psi^*$.  Moreover, if $\varphi\in \C$ is proper,  
the Fenchel-Moreau theorem  (or Fenchel biconjugation theorem) \cite{BV10, Roc70} yields that  $(\varphi^*)^*=\varphi$ if and only if $\varphi$  is lower semi-continuous.  It is obvious by \eqref{def-dual-1}  that  $\varphi ^*(0) = - \inf_{x\in \R^n} \varphi(x)$ for $\varphi\in \C$, and hence $\varphi^*$ is proper if $\inf _{x\in \R^n} \varphi(x) > - \infty$. Furthermore, if $\varphi\in \C$ is proper,  then $\varphi^*(y)>-\infty$ for all $y\in\R^n$.  We also let $e^{-\infty}=0$ by default (unless otherwise stated). 

Denote by $\nabla \varphi(x)$ the gradient of $\varphi$ at $x\in \dom(\varphi)$ if $\varphi$ is differentiable at $x$.  A convex function $\varphi\in \C$  may not be differentiable at each point in its domain, however it is always continuous on $\R^n\setminus \partial (\dom (\varphi))$ and is  differentiable  almost everywhere in $\mathrm{int} (\dom (\varphi))$.  It can be checked that if $\varphi$ is differentiable at $x$, then the supremum in  \eqref{def-dual-1} achieves its maximum if $y=\nabla \varphi(x)$. Hence, \begin{equation}
\varphi^*(\nabla \varphi(x))= \langle x , \nabla \varphi(x) \rangle -\varphi(x). \label{def-dual-2}
\end{equation}  This relation will be used often in later context. Denote by $\nabla^2 \varphi$ the Hessian matrix of $\varphi$.

Throughout the paper,  we shall consider two classes of functions $\L$ and $\A$ given by  
\begin{eqnarray*} 
	\L &=& \Big\{ \varphi\in \C \ \big| \ \varphi \ \mathrm{is\ proper \ such\ that}\  \lim_{|x|\rightarrow+\infty}\varphi(x)=+\infty \Big\},
	\\ 
	\A& =&\Big\{f: \R^n\to\R \ \big| \ f=e^{-\varphi}, \ \ \varphi\in\L \Big\}.
\end{eqnarray*} It has been proved in \cite[Lemma 2.5]{CF13} that if $\varphi \in \L$, then for all $x\in\R^n$, 
\begin{equation}\label{claimab}
\varphi(x)\geq a|x|+b,
\end{equation} where  $a>0$ and $b$ are two constants, and
$\varphi ^*$ is proper and satisfies $\varphi ^*(y) > - \infty$  for all $y \in \R^n$.

A function $f: \R^n\rightarrow \R$ is said to be log-concave if $f=e^{-\varphi}$ for some convex function $\varphi\in \C$.   We are interested in log-concave functions such that $0<J(f)<\infty$, where $J(f)$ is the total mass of $f$ defined in \eqref{totalmass}. It can be easily checked (see e.g.~\cite[p.~3835]{CK15}) that, for $f=e^{-\varphi}$ a log-concave function,   $0<J(f)<\infty$  is equivalent to the facts that $\mathrm{dom}(\varphi)$ is not contained in a hyperplane and  $\lim_{|x|\rightarrow+\infty}\varphi(x)=+\infty$ (in particular, $f\in \A$). It follows from   \eqref{def-dual-1} and  \eqref{claimab} that, if $f=e^{-\varphi}$ with $\varphi\in \C$ is non-degenerate (i.e., not vanishing on a set with positive Lebesgue measure), then  $J(f)<\infty$ if and only if $o\in \mathrm{int}(\dom(\varphi^*))$.

The surface area measure for a log-concave function $f$   (or the {\it moment measure} of $-\log f$ as in \cite{CK15})  plays important roles in applications. The following definition follows from \cite[Definition 1]{CK15} (see also \cite[Definition 4.1]{CF13}).  
\begin{definition} \label{def-moment-measure}  Let $f=e^{-\varphi}$  
 be  a  log-concave function. The surface area measure of $f$ (or the moment measure of $\varphi$), denoted by $\mu(f, \cdot)$, is the Borel measure on $\R^n$ such that $\mu (f,\cdot)$ is the push-forward measure of $e^{-\varphi(x)}\,dx$ under $\nabla \varphi$, that is, 
	\begin{equation}\label{moment-form-1} \int_{\R^n}g(y)\, d\mu(f, y)=\int_{\R^n}g(\nabla \varphi(x))e^{-\varphi(x)} \,dx, \end{equation}  for every Borel function $g$ such that $g \in L ^ 1 (\mu (f,\cdot))$  or $g$ is non-negative. \end{definition}		
It has been proved that $\varphi ^* \in L ^ 1 (\mu (f,\cdot))$, if $f = e ^ {-\varphi} \in \A$ with $\varphi^*\geq 0$ in \cite[Lemma 4.12]{CF13}  or  
if $0<J(e^{-\varphi})<\infty$ in \cite[Proposition 7]{CK15}, i.e.,   \begin{equation} -\infty<\int_{\R^n} \varphi^* (y)  \, d\mu(f, y)=\int_{\R^n} \varphi^* (\nabla \varphi(x) ) \, e^{-\varphi(x)} \, dx < + \infty.\label{finitezza-1} \end{equation}   Let $(\varphi \alpha)(x) =\alpha \varphi(\frac{x}{\alpha})$ for $\alpha>0$ and $\varphi 0 =\mathbf{I}_{\{0\}}$ 
where $\mathbf{I}_E$ is the indicatrix function of $E$ (i.e., $\mathbf{I}_E(x)=0$ if $x\in E$ and $\mathbf{I}_E=+\infty$ if $x\notin E$). By \eqref{const-mul-dual}, one sees  $(\alpha \varphi)^*=  \varphi^*\alpha $ for $\alpha>0$.   
It can be easily checked from \eqref{A-sum-6-23} that, for  real $\alpha, \beta>0$ and two log-concave functions $f=e^{-\varphi}$ and $g=e^{-\psi}$,  
the  Asplund sum of $\alpha \cdot f\oplus \beta\cdot g$ may  also be formulated by \begin{equation}\label{A-sum-1}
\alpha\cdot f \oplus \beta\cdot g= e^{-\varphi \alpha \Box \psi  \beta},
\end{equation}  
where $\varphi \Box \psi$ is the {\it infimal convolution}  of $\varphi$ and $\psi$ defined by $$
\varphi \Box \psi (x) = \inf _{y\in \R^n} \big \{ \varphi (x-y) + \psi (y) \big \}, \ \ \mathrm{for\ any}\ x\in \R^n.$$ It can be easily checked that 
\begin{equation}\label{inf-conv-1}
(\varphi \alpha \Box \psi  \beta)^* =\alpha \varphi^*+\beta \psi^*.
\end{equation}   
The   {\it Pr\'ekopa-Leindler inequality}  implies  that (see, e.g., \cite[Remark 3.3]{CF13}), for any integrable log-concave functions  $f, g$  and $0<\lambda<1$, one has  
\begin{eqnarray}\label{PLineq}
J\big((1-\lambda)\cdot f\oplus \lambda\cdot g\big) \geq J(f)^{1-\lambda} J(g)^{\lambda},
\end{eqnarray} with equality  if and only if there exists   $x_0\in\R^n$ such that  $f(x)=g(x-x_0)$  (see, e.g., \cite{Dub77,Gar02}). 

\section{The $L_{p}$  Asplund sum of log-concave functions  
}\label{p-sum of log-concave function} \setcounter{equation}{0}
This section dedicates to the study of properties of the $L_{p}$   Asplund sum of log-concave  functions. In particular,   a Pr\'ekopa-Leindler  type inequality related to the $L_{p}$   Asplund sum will be proved.  We  focus on $\L_0\subset \L$ and $\A_0\subset \A$, where  $\A_0=\big\{ e^{-\varphi}:   \varphi\in\L_0 \big\}$ with  \begin{align*}
\L_0&=\Big\{ \varphi \in\L:  \ \varphi \  \text{is\ non-negative and  lower semi-continuous such that}\ \varphi(o)=0  \Big\}.
\end{align*}  Clearly, if $\varphi\in \L_0$, then $\varphi$ has its minimum attained at the origin $o$.   	
From \eqref{def-dual-1}, one sees that if $\varphi(o)=0$, then  \begin{equation}
\varphi^*(y)=\sup_{x\in\R^n}\big\{ \langle x , y \rangle -\varphi(x)\big\}\geq  \langle o , y \rangle -\varphi(o)=0,   \ \ \mathrm{for\ all\ } y \in \R ^n. \label{non-negative-6-26}\end{equation} 	Moreover, $\varphi^*(o)=0$ as  \begin{equation}\label{dual-at-origin} 
\varphi^*(o)=\sup_{x\in\R^n} -\varphi(x) =-\inf_{x\in \R^n} \varphi(x)=0. \end{equation}  

The definition for the $L_{p}$  Asplund sum of log-concave functions is given below, with concentration on $p>1$.

\begin{definition}\label{def-lp-A-sum} 
	Let $f=e^{-\varphi}\in \A_0$ and $g=e^{-\psi}\in\A_0$. For $p>1$, define  $ f\oplus_{p} g$, the $L_p$ Asplund sum of  $f$ and $g$, by  $ f\oplus_{p}  g=e^{-\varphi  \square_{p}\psi}$, where   	
	\begin{equation}\label{prodottoL-no constants} 
	\varphi\square_{p} \psi=\left[
	\big((\varphi^*)^p+(\psi^*)^p\big)^{\frac{1}{p}}\right]^*.
	\end{equation}
\end{definition} 
For convenience, if $\alpha>0$, let 
\[(\varphi \cdot_p \alpha) (x)=( \varphi  \alpha^{\frac{1}{p}})(x)=\alpha^{\frac{1}{p}} \varphi (\alpha^{-\frac{1}{p}} x) \ \ \mathrm{for\ any}\ x\in \R^n.\] For $\alpha,\beta> 0$, $p\geq1$,  $f=e^{-\varphi}\in \A_0$ and $g=e^{-\psi}\in\A_0$, let  $\alpha\cdot_p f\oplus_{p}\beta\cdot_p g=e^{-\varphi\cdot_p\alpha \square_{p}\psi\cdot_p\beta},$ 	where  	
\begin{equation}\label{prodottoL} 
\varphi\cdot_p\alpha \square_{p}\psi\cdot_p\beta=\left[
\big(\alpha(\varphi^*)^p+\beta(\psi^*)^p\big)^{\frac{1}{p}}\right]^*.
\end{equation}	
By \eqref{inf-conv-1} and \eqref{prodottoL},    $\alpha\cdot_p f\oplus_{p}\beta\cdot_p g$  for $p=1$ reduces to the Asplund sum $\alpha\cdot f\oplus\beta\cdot g$.  		

The following result asserts that   the $L_p$ Asplund sum of log-concave functions is closed in $\A_0.$
\begin{proposition}\label{close-Lp-A-sum}
	Let $f=e^{-\varphi}\in \A_0$ and $g=e^{-\psi}\in\A_0$.  For $p>1$ and $\alpha, \beta>0$, let \begin{equation}\label{relation-p-sq-1} \phi=\big[\alpha(\varphi^*)^p+\beta(\psi^*)^p\big]^{\frac{1}{p}}.\end{equation}   Then  $\phi= \big(\varphi\cdot_p\alpha \square_{p}\psi\cdot_p\beta\big)^*$
	and  $\alpha\cdot_p f\,\oplus_p\,\beta\cdot_p g=e^{-\phi^*}\in\A_0.$ 
\end{proposition}
\begin{proof}  Let $f=e^{-\varphi}\in \A_0$ and $g=e^{-\psi}\in\A_0$ which imply  $\varphi, \psi \in\L_0$. By \eqref{non-negative-6-26} and \eqref{dual-at-origin},  $\varphi^*\geq 0$ and $\psi^*\geq 0$ with $\varphi^*(o)=\psi^*(o)=0$.  Hence,  $ \varphi\cdot_p\alpha \square_{p}\psi\cdot_p\beta$ for $p> 1$ given in  \eqref{prodottoL} is well-defined and $\varphi\cdot_p\alpha \square_{p}\psi\cdot_p\beta\in \C$ is lower semi-continuous. Moreover, $\phi$ defined by \eqref{relation-p-sq-1}  is non-negative and  $\phi(o)=0$ is the minimum of $\phi$. This further yields that   $ (\varphi\cdot_p\alpha \square_{p}\psi\cdot_p\beta)(o)= 0$ by \eqref{dual-at-origin}, and hence  is proper.  Note that $\varphi^*, \psi^*\in \C$ are lower semi-continuous. Consequently, for $p>1$, $\alpha, \beta>0$, $\phi$ is lower semi-continuous. Moreover,   $\phi\in \C.$ To this end,  for all $\lambda\in (0, 1)$ and $x, y\in \dom(\phi)$, the Minkowski inequality for norms  yields that   \begin{eqnarray*} \phi(\lambda x+(1-\lambda)y)\!\! 
		&=&\! \! \Big[\alpha\big(\varphi^*(\lambda x+(1-\lambda)y) \big)^p+\beta\big(\psi^*(\lambda x+(1-\lambda)y)\big)^p \Big]^{\frac{1}{p}}\\ \!\!&\leq&\!\! \Big[\alpha\big(\lambda \varphi^*(x)+(1-\lambda)\varphi^*(y) \big)^p+\beta\big(\lambda \psi^*(x)+(1-\lambda)\psi^*(y) \big)^p \Big]^{\frac{1}{p}} \\ \!\! &\leq& \!\! \lambda  \big[\alpha(\varphi^*)^p+\beta(\psi^*)^p\big]^{\frac{1}{p}}(x) + (1-\lambda) \big[\alpha(\varphi^*)^p+\beta(\psi^*)^p\big]^{\frac{1}{p}}(y)\\ \!\! &=&\!\! \lambda  \phi(x) + (1-\lambda) \phi(y).\end{eqnarray*}  According to \eqref{prodottoL} and  the Fenchel-Moreau theorem \cite{BV10, Roc70}, one sees that 
	\begin{equation} \label{biconjugate-1}  \big(\varphi\cdot_p \alpha\square_{p}\psi\cdot_p \beta\big)^*=\big(\phi^*\big)^*=\phi \ \  \ \mathrm{and}\  \ \ \alpha\cdot_p f\,\oplus_p\,\beta\cdot_p g=e^{-\phi^*}.
	\end{equation}

	To claim that $\alpha\cdot_p f\,\oplus_p\,\beta\cdot_p g\in\A_0$  for $p>1$, it is enough to show that $\phi^*\in\L_0$.  To this end, we only need  to claim that   $\lim_{|x|\rightarrow+\infty} \phi^*(x)=+\infty$ as  we already showed that $\phi^*$ is proper and lower semi-continuous.  Inequality  \eqref{claimab} implies that  for all $x\in \R^n$,  $\varphi(x)\geq a|x|+b$  and $\psi(x)\geq a'|x|+b'$  for some $a, a'>0$ and $b,b'\in \R$.
	Let $c=\min\{a,a'\}>0$ and $d=\min\{b, b'\}$. Hence $\varphi(x)\geq c|x|+d$ and  $\psi(x)\geq c|x|+d$ for all $x\in \R^n$.  This further yields that $\varphi ^*\leq (c|x|+d)^*$  and $\psi^*\leq (c|x|+d)^*$. Accordingly, due to $p>1$ and $\alpha, \beta>0$, one has 
	$$\phi=\big[\alpha(\varphi^*)^p+\beta(\psi^*)^p\big]^{\frac{1}{p}} \leq \big[\alpha((c|x|+d)^*)^p+\beta((c|x|+d)^*)^p\big]^{\frac{1}{p}}= (\alpha+\beta)^{\frac{1}{p}} (c|x|+d)^*.$$ It follows from \eqref{def-dual-1} and \eqref{const-mul-dual} that for all $y\in \R^n$,
	\begin{equation*}  \phi^*(y)\geq  \big[(\alpha+\beta)^{\frac{1}{p}} (c|x|+d)^*\big]^*(y)  =  c|y| +d  (\alpha+\beta)^{\frac{1}{p}}.\end{equation*}
	In particular,  $\lim_{|y|\rightarrow+\infty} \phi^*(y)=+\infty$. 	This concludes that 	 $\alpha\cdot_p f\,\oplus_p\,\beta\cdot_p g=e^{-\phi^*}\in\A_0.$
\end{proof}

Let $\chi_E$ be the characteristic function of $E$, i.e., $\chi_E(x)=1$ if $x\in E$ and $\chi_E(x)=0$ if $x\notin E$. 		
It is well-known that
if $K\in \K_{(o)}^n$, 
then $\chi_K$ is a log-concave function with $\mathbf{I}_K=-\log(\chi_K)\in \L_0$ and hence $\chi_K\in \A_0$. By \eqref{support-1} and \eqref{def-dual-1}, one sees that $ (\mathbf{I}_K)^*=h_K$ and  $(h_K)^*=\mathbf{I}_K$. Together with \eqref{Lp-addition}, for two convex bodies $K, L\in \K_{(o)}^n$,  one has, for $p\geq 1$, \begin{equation*}
(\mathbf{I}_K)\cdot_p \alpha \square_{p} (\mathbf{I}_L)\cdot_p \beta=\Big(
\big(\alpha((\mathbf{I}_K)^*)^p+\beta((\mathbf{I}_L)^*)^p\big)^{\frac{1}{p}}\Big)^*=\Big(
\big(\alpha h_K^p+\beta h_L^p\big)^{\frac{1}{p}}\Big)^*=\mathbf{I}_{\alpha \cdot_p K+_p\beta\cdot_p L}. 
\end{equation*}  
Therefore, for $\alpha, \beta>0$ and $p\geq 1$,  $$	\alpha\cdot_p\chi_K\oplus_p\beta\cdot_p\chi_L=e^{-(\mathbf{I}_K)\cdot_p \alpha \square_{p} (\mathbf{I}_L)\cdot_p \beta}=\chi_{\alpha \cdot_p K+_p\beta\cdot_p L}.$$ Moreover,  for two convex bodies $K, L\in \K_{(o)}^n$ and for any $\alpha, \beta>0$, 
\begin{eqnarray}
J(\alpha\cdot_p \chi_K\oplus_p \beta\cdot_p \chi_L) 
=\int_{\R^n}  \chi_{\alpha\cdot_p K+_p \beta\cdot_p L} (x)\,dx = V(\alpha \cdot_p K+_p \beta\cdot_p L). \label{volume-p-addition-515} \end{eqnarray} By the $L_p$ Brunn-Minkowski inequality \eqref{p-BM},   for $p\geq 1$, $\lambda\in (0, 1)$, and $K, L\in \K_{(o)}^n$,  one has \begin{eqnarray}\label{p-BM-2-2}
V\big((1-\lambda)\cdot_p K+_p \lambda\cdot_p L\big)\geq \Big((1-\lambda) V(K)^{\frac{p}{n}}+\lambda V( L)^{\frac{p}{n}}\Big)^{\frac{n}{p}}\geq V(K)^{1-\lambda}V(L)^{\lambda},
\end{eqnarray} where the last inequality follows from the fact that the logarithmic function is concave. Combining \eqref{volume-p-addition-515}  and \eqref{p-BM-2-2}, one easily sees that, for $p\geq 1$, $\lambda\in (0, 1)$, and $K, L\in \K_{(o)}^n$,  \begin{eqnarray*} J\Big((1-\lambda)\cdot_p \chi_K\oplus_p \lambda\cdot_p \chi_L\Big)&\geq & J(\chi_K)^{1-\lambda}J(\chi_L)^{\lambda}. \end{eqnarray*}  This is a Pr\'ekopa-Leindler type inequality  and can be extended to all $f, g\in \A_0$.  This result is proved in the following theorem, where $p>1$ is concentrated as the case $p=1$ reduces to the classical  {\it Pr\'ekopa-Leindler inequality} \eqref{PLineq}.  

\begin{theorem}\label{prekopainequality} Let  $f=e^{-\varphi}\in \A_0$ and $g=e^{-\psi} \in\A_0$. For $\lambda\in (0,1)$ and $p>1$,  it holds that
	\begin{eqnarray}\label{prekopainequality1}
	J\big((1-\lambda)\cdot_p f\oplus_{p} \lambda\cdot_p g\big)\geq J(f)^{1-\lambda}J(g)^{\lambda}
	\end{eqnarray}
	with  equality if and only if $f=g$ on $\R^n$. 
\end{theorem}
\begin{proof} Let  $f=e^{-\varphi}\in \A_0$ and $g=e^{-\psi} \in\A_0$.   For $\lambda\in (0,1)$ and $p>1$,  let  
	$
	(1-\lambda )\cdot_p f\oplus_{p} \lambda \cdot_p g=e^{-\phi_{\lambda}^*}.
	$ According to Proposition \ref{close-Lp-A-sum},  one has $(\phi_{\lambda}^*)^*=\phi_{\lambda}$ and  $\phi_{\lambda }=\big[(1-\lambda)(\varphi^*)^p+\lambda (\psi^*)^p\big]^{\frac{1}{p}}.$  Since the function $t^{\frac{1}{p}}$ is  strictly concave  on $t\in [0, \infty)$ when $p>1$, one has,  \begin{equation}\label{p-conv-equality-515} \phi_{\lambda }=\big[(1-\lambda)(\varphi^*)^p+\lambda (\psi^*)^p\big]^{\frac{1}{p}}  \geq(1-\lambda) \varphi^*+\lambda \psi^*. \end{equation} Together with \eqref{inf-conv-1}, this in turn implies that 
	\begin{eqnarray}\label{relationship of p and 1 infimal convolution}
	\phi_{\lambda}^*  \leq \big((1-\lambda) \varphi^*+\lambda \psi^* \big)^*  
	=  \varphi (1-\lambda) \Box  \psi  \lambda.
	\end{eqnarray} By the classical Pr\'ekopa-Leindler inequality (\ref{PLineq}),  inequality \eqref{prekopainequality1} holds: 
	\begin{eqnarray} 
	J\big((1-\lambda)\cdot_p f\oplus_{p}\lambda \cdot_p g\big)= \int _{\R^n} e^{-\phi_{\lambda}^*(x)}\,dx \geq   \int _{\R^n} e^{-\big(\varphi (1-\lambda) \Box  \psi  \lambda\big)(x)}\,dx 		\geq   J(f)^{1-\lambda}J(g)^{\lambda}.\label{p-prekopa-l}
	\end{eqnarray} 
	
	Now let us characterize the equality condition.  It is obvious  that equality holds if $f=g$. On the other hand, to have equality in  \eqref{prekopainequality1}, equalities must occur  in \eqref{p-prekopa-l}, and consequently equalities hold for the classical Pr\'ekopa-Leindler inequality (\ref{PLineq}) and for \eqref{relationship of p and 1 infimal convolution} (equivalently for \eqref{p-conv-equality-515}). The former one implies that there exists   $x_0\in\R^n$ such that  $f(x)=g(x-x_0)$, which in turn yields $\varphi(x)=\psi(x-x_0)$. Combining with the latter one and the fact that $t^{\frac{1}{p}}$ is  strictly concave, one gets  $\varphi^*=\psi^*$ and hence $f=g$ as desired. \end{proof}

\begin{corollary} \label{concave-PL-1} Let  $f=e^{-\varphi}\in \A_0$, $g=e^{-\psi} \in\A_0$ and $p>1$. Then, $J(f\oplus _p t\cdot_p g)$ is log-concave on $t\in (0, \infty)$ and $J\big((1-t)\cdot_p f\oplus_p t\cdot_p g\big)$ is log-concave on $t\in (0, 1)$. 
\end{corollary} \begin{proof} 
	Let  $f=e^{-\varphi}\in \A_0$ and $g=e^{-\psi} \in\A_0$.  For $\lambda\in (0, 1)$ and $t, s\in (0, \infty)$, let $\eta=(1-\lambda)t+\lambda s$.  According to Proposition \ref{close-Lp-A-sum}, one gets  $f\oplus_p \eta \cdot_p g=e^{-\phi^*_{\eta}}$  with $\phi_{\eta}=\big[(\varphi^*)^p+\eta(\psi^*)^p\big]^{\frac{1}{p}}.$
	It follows from Proposition \ref{close-Lp-A-sum} and the fact that the function $t^{\frac{1}{p}}$ is concave on $t\in (0, \infty)$ for $p>1$ that 
	\begin{align*} 
		\phi_{\eta}
		&= \big[\big((1-\lambda) (\varphi^*)^p+t (\psi^*)^p\big)+\lambda \big( (\varphi^*)^p+ s (\psi^*)^p\big)\big]^{\frac{1}{p}} \\
		&\geq (1-\lambda)   \big((\varphi^*)^p+t (\psi^*)^p\big)^{\frac{1}{p}} +\lambda  \big( (\varphi^*)^p+ s (\psi^*)^p\big)^{\frac{1}{p}}\\ 
		&=  (1-\lambda)  (\phi^*_{t})^*+\lambda  (\phi^*_{s})^*.\end{align*}   
	By \eqref{A-sum-1} and \eqref{inf-conv-1}, one gets   $ \phi^*_{\eta}\leq \big((1-\lambda)  (\phi^*_{t})^*+\lambda  (\phi^*_{s})^*\big)^*=  \phi^*_{t}(1-\lambda) \Box   \phi^*_{s} \lambda $
	and hence, $$f\oplus_p \eta \cdot_p g=e^{-\phi^*_{\eta}}\geq (1-\lambda)\cdot \big(f\oplus_p t \cdot_p g\big)\oplus \lambda\cdot \big(f\oplus_p s \cdot_p g\big).$$ By the 
	classical  Pr\'ekopa-Leindler inequality \eqref{PLineq}, one gets the desired log-concavity for $J(f\oplus _p t\cdot_p g)$ on $t\in (0, \infty)$, that is, 
 $$\log J(f\oplus_p \eta \cdot_p g) \geq (1-\lambda) \log  J(f\oplus_p t \cdot_p g)+\lambda\log J(f\oplus_p s \cdot_p g).$$  	
	
	The log-concavity for $J\big((1-t)\cdot_p f\oplus _p t\cdot_p g\big)$  on $t\in (0, 1)$   is a direct result of Theorem \ref{prekopainequality}.  
\end{proof}

\section{An $L_p$ Minkowski type inequality} \label{sectoion-4-6-23} \setcounter{equation}{0}

In this section, we propose a definition for $\delta J_{p}(f, g)$,  the first variation of the total mass at $f$ along $g$ with respect to the $L_{p}$  Asplund sum. A Minkowski type inequality for $\delta J_{p}(f, g)$ will be established.   Our definition for $\delta J_{p}(f, g)$ is given below.   
\begin{definition}\label{variation-p-5-15-1} 
	Let $f, g  \in\A_0$. For $p>1,$  define $\delta J_{p}(f, g)$ by   
	\[
	\delta J_{p}(f,g)=\underset{t\rightarrow0^+}{\lim}\frac{J(f\oplus_{p} t\cdot_p g)-J(f)}{t}
	\]
	whenever the  limit exists. 
\end{definition}

Let $J(f)>0$ and $\mathrm{Ent}(f)$ be the entropy of $f$ which  may be formulated by \begin{equation}\label{def-entropy}
\mathrm{Ent}(f)=\int_{\R^n}f\log f\,dx-J(f)\log J(f).  
\end{equation}  
According to \cite[Proposition 3.11]{CF13}, $\mathrm{Ent} (f)$ is finite if $f\in \A$   such that  $J(f)>0$. 
Our main result in this section is the following Minkowski type inequality for $\delta J_{p}(f, g)$.  We only focus on $p>1$ as  inequality \eqref{mink1}  and its equality condition for $p=1$
have already been proved  in  \cite[Theorem 5.1]{CF13}.

\begin{theorem}\label{minkowskiinequality}
	Let $f, g \in\A_0$ such that $J(f)>0$ and $J(g)>0$.  For $p> 1,$ one has, 
	\begin{equation}\label{mink1}
	\delta J_{p}(f,g)\geq J(f)\left[ \frac{n}{p} +\frac{1-p}{p}\log J(f)+\log J(g)\right]+\frac{1}{p}{\rm Ent}(f),
	\end{equation}
	with equality if and only if  $f=g$. \end{theorem}

	  In order to prove Theorem \ref{minkowskiinequality}, we shall need some preparation. The following result shows how to calculate  $\delta J_{p}(f, f)$. Again, the case for $p=1$ has been covered in \cite[Proposition 3.11]{CF13} and will not be repeated in the following result. From Lemma \ref{specialcaseoffg}, one sees that $\delta J_{p}(f,f)$ is finite if $f\in \A$   such that  $J(f)>0$.

\begin{lemma}\label{specialcaseoffg} Let $f\in \A_0$  such that $J(f)>0$.  For $p>1,$ one has 
	\begin{equation}\label{entropu-p>1}
	\delta J_{p}(f,f)= \frac{n}{p} J(f)   + \frac{1}{p} \int_{\R^n} f \log f \, dx=\frac{n+\log J(f)}{p} J(f)+\frac{1}{p} \mathrm{Ent}(f).
	\end{equation} 
\end{lemma}
\begin{proof} Let $f=e^{-\varphi}\in \A_0$ such that $J(f)>0$.  It can be checked from \eqref{prodottoL}  that, for $p> 1$  and $t>0$,  $f\oplus_{p}t\cdot_p f=e^{-\phi_t}$ with   $$ \phi_t(x)=[\varphi\cdot_p (1+t)](x)=(1+t)^{\frac{1}{p}}\varphi\bigg(\frac{x}{(1+t)^{\frac{1}{p}}}\bigg).$$  Consequently,  by letting $x=(1+t)^{\frac{1}{p}}y$, one gets 
	\begin{align*}
	\delta J_{p}(f,f)  &=\lim_{t\rightarrow0^{+}}\frac{J(f\oplus_{p}t\cdot_p f)-J(f)}{t}\\
	& =\lim_{t\rightarrow 0^{+}} \frac{1}{t}  \bigg((1+t)^{\frac{n}{p}}  \int_{\R^n} e^{-(1+t)^{\frac{1}{p}}\varphi(y)}\,dy - \int_{\R^n} e^{-\varphi(y)} \,dy\bigg) \\ 
	& =\lim_{t\rightarrow 0^{+}} \frac{(1+t)^{\frac{n}{p}}-1}{t}  \int_{\R^n} e^{-(1+t)^{\frac{1}{p}}\varphi(y)}\,dy  +\lim_{t\rightarrow 0^{+}}  \int_{\R^n}  \frac{e^{-(1+t)^{\frac{1}{p}}\varphi(y)}-e^{-\varphi(y)}}{t} \,dy. \end{align*}
	Note that $\varphi\in \L_0$ is non-negative. It follows from the monotone convergence theorem that 
	\begin{eqnarray*} \lim_{t\rightarrow 0^{+}} \frac{(1+t)^{\frac{n}{p}}-1}{t}  \int_{\R^n} e^{-(1+t)^{\frac{1}{p}}\varphi(y)}\,dy=\frac{n}{p}  \lim_{t\rightarrow 0^{+}} \int_{\R^n} e^{-(1+t)^{\frac{1}{p}}\varphi(y)}\,dy =\frac{n}{p}J(f). 
	\end{eqnarray*}
	Similarly,  one can also have 
	\begin{eqnarray*} \lim_{t\rightarrow 0^{+}} \! \int_{\R^n} \! \frac{e^{-(1+t)^{\frac{1}{p}}\varphi(y)}-e^{-\varphi(y)}}{t} \,dy =\!\int_{\R^n} \lim_{t\rightarrow 0^{+}}   \frac{e^{-(1+t)^{\frac{1}{p}}\varphi(y)}-e^{-\varphi(y)}}{t} \,dy=-\frac{1}{p}\int_{\R^n}\varphi(x)e^{-\varphi(x)}\,dx.
	\end{eqnarray*}
	By \eqref{def-entropy}, one gets 
	\[ \delta J_{p}(f,f)= \frac{n}{p}J(f)   +\frac{1}{p}\int_{\R^n} f \log f  \, dx=  \frac{n}{p}J(f)   + \frac{1}{p} \mathrm{Ent}(f)+ \frac{1}{p}  J(f)\log J(f).\] 
	This is exactly the desired equality \eqref{entropu-p>1}. 
\end{proof}

Our second lemma is to extend \cite[Lemma 3.9]{CF13} for $p=1$ to all $p>1$. 

\begin{lemma}\label{boundness} Let $f=e^{-\varphi}\in \A_0$ and $g=e^{-\psi} \in\A_0$. For $p> 1$  and for $t> 0$, set \begin{equation} \label{varphi-t} \varphi_{t}= \varphi\square_{p}( \psi \cdot_p t) \end{equation} and $f_{t}=e^{-\varphi_{t}}$. Then,  for any $ x\in\mathbb{R}^n$ and $t, s\in(0,1]$ such that $s<t$, one has, 
	\[\varphi_{1}(x)\leq \varphi_{t}(x)\leq \varphi_s(x)\leq  \varphi (x)  \quad \mathrm{and} \quad  f(x)\leq f_s(x)\leq  f_{t}(x)\leq f_{1}(x).  \] \end{lemma}
\begin{proof}
	Let $t > 0$, $\delta>0,$ and  $p\geq1$. Note that   $f=e^{-\varphi}\in \A_0$ and $g=e^{-\psi} \in\A_0$ imply $\varphi^*, \psi^*\geq 0$. By Proposition \ref{close-Lp-A-sum},  \eqref{relation-p-sq-1} (or \eqref{biconjugate-1}) and  \eqref{varphi-t}, one gets  that \begin{equation}\label{varphi-t-star}  \varphi_t^*= \big((\varphi^*)^p+t(\psi^*)^p\big)^{\frac{1}{p}}\geq \varphi^*. \end{equation}  Clearly, $\varphi_{t+\delta}^*\geq \varphi_t^*$ and $\varphi_{t+\delta}=(\varphi_{t+\delta}^*)^*\leq (\varphi_t^*)^*=\varphi_t.$  Moreover, for any $ x\in\mathbb{R}^n$ and $t\in[0,1]$,  \[
	\varphi_{1}(x)\leq \varphi_{t}(x)\leq \varphi (x)  \quad \mathrm{and} \quad  f(x)\leq f_{t}(x)\leq f_{1}(x).  \]  This completes the proof of this lemma. 
\end{proof}

The following lemma proves that $\delta J_p (f, g)$ indeed exists (although may be $+\infty$). Again we only focus on $p>1,$ as $p=1$ has been covered in \cite[Theorem 3.6] {CF13}. 
\begin{lemma}\label{diffe} Let $f, g  \in\A_0$ such that $J(f)>0$.   For $p> 1,$ one has $
	\delta J_{p}(f,g) \in[0, + \infty].$   \end{lemma}

\begin{proof} Let $f=e^{-\varphi}\in \A_0$ and $g=e^{-\psi} \in\A_0$ such that $J(f)>0$. Then  $\varphi, \psi \in \L_0$. Let $\varphi_{t}$ be as in   \eqref{varphi-t}.   It follows from Lemma \ref{boundness} that $\varphi(x) \geq\bar{\varphi}(x):= \lim_{t\rightarrow 0^+} \varphi_t(x)$ for every $x\in \R^n$  and $$
	J(e^{-\bar{\varphi}})=\lim_{t\rightarrow 0^+} J(e^{-\varphi_t})\geq J(e^{-\varphi}),
	$$ by the monotone convergence theorem. Note that \begin{eqnarray} \delta J_p(f, g)=
	\lim_{t\rightarrow0^{+}}\frac{J(e^{-\varphi_t})-J(e^{-\varphi})}{t}. \label{replace-1}
	\end{eqnarray} Therefore, $\delta J_p(f, g)=+\infty$  if $J(e^{-\bar{\varphi}})>J(e^{-\varphi})$,  and $\delta J_p(f, g)=0$ if $J(e^{-\varphi_{t_0}})=J(e^{-\varphi})$ for some $t_0>0$ (hence $J(e^{-\bar{\varphi}})=J(e^{-\varphi_{t}})=J(e^{-\varphi})$ for every $t\in [0,t_0]$  due to Lemma \ref{boundness}).

	Lastly, we consider the case that $J(e^{-\varphi_t})>J(e^{-\varphi})$  for all $t>0$ but  $J(e^{-\bar{\varphi}})=J(e^{-\varphi})$.  Note that $J(f)=J(e^{-\varphi})>0.$ In this case,  one has  
	\begin{eqnarray} 
	\frac{J(e^{-\varphi_{t}})-J(e^{-\varphi})}{t}=\frac{\log J(e^{-\varphi_{t}})-\log J(e^{-\varphi})}{t}\cdot\frac{J(e^{-\varphi_{t}})-J(e^{-\varphi})}{\log J(e^{-\varphi_{t}})-\log J(e^{-\varphi})}.\label{split-1}
	\end{eqnarray} Corollary  \ref{concave-PL-1} and Lemma \ref{boundness} imply that  $\log J(e^{-\varphi_{t}})$ is  an  increasing and concave function on $t\in (0, \infty)$. Thus, \begin{eqnarray} 
	J(e^{-\varphi}) =\lim_{t\rightarrow 0^{+}}\frac{J(e^{-\varphi_{t}})-J(e^{-\varphi})}{\log J(e^{-\varphi_{t}})-\log J(e^{-\varphi})}, \ \ \mathrm{and} \ \ \lim_{t\rightarrow 0^{+}}\frac{\log J(e^{-\varphi_{t}})-\log J(e^{-\varphi})}{t}\in [0,+\infty].  \label{3.2}\end{eqnarray} Combining \eqref{split-1} and \eqref{3.2}, one gets $\delta J_p(f,g)\in[0,+\infty]$, and this completes the proof. 
\end{proof}

We also need the following lemma. The case for $p=1$ has been given in \cite[Lemma 5.4]{CF13}, so we only state the result for $p>1$. 

\begin{lemma}\label{lemma35}
	Let $f=e^{-\varphi}\in \A_0$ and $g=e^{-\psi} \in\A_0$ such that $J(f)>0$.  For $p> 1$, one has, \begin{equation}\label{decompose-1} 
	\lim_{t\rightarrow0^+} \frac{J((1-t)\cdot_p f\oplus_{p} t\cdot_p g)-J(f)}{t}=\delta J_{p}(f,g)-\delta J_{p}(f,f).
	\end{equation} 
\end{lemma}
\begin{proof}  Let $f=e^{-\varphi}\in \A_0$ and $g=e^{-\psi} \in\A_0$ such that $J(f)>0$.  According to Proposition \ref{close-Lp-A-sum}, for $p> 1$  and $t\in(0,1)$,   $(1-t)\cdot_p f\oplus_{p}t\cdot_p g=e^{-\phi_t^*}$ where $$\phi_t= \big((1-t)(\varphi^*)^p+t(\psi^*)^p\big)^{\frac{1}{p}}= (1-t) ^{\frac{1}{p}} \Big((\varphi^*)^p+\frac{t}{1-t} (\psi^*)^p\Big)^{\frac{1}{p}}.$$
	It can be checked by \eqref{const-mul-dual} that \begin{eqnarray*} \phi_t ^*(x) =(1-t)^{\frac{1}{p}}\Big[\big((\varphi^*)^p+\frac{t}{1-t}(\psi^*)^p \big)^{\frac{1}{p}}\Big]^*\Big(\frac{x}{(1-t)^{1/p}}\Big).
	\end{eqnarray*} By letting $x=(1-t)^{\frac{1}{p}} y$ and  $s=\frac{t}{1-t}$, one gets  $1-t=\frac{1}{1+s}$ and 
	\begin{align*}
	J((1-t)\cdot_p f\oplus_{p}t\cdot_p g) 
	& =  \int_{\R^n}  (1-t)^{\frac{n}{p}} e^{-(1-t)^{\frac{1}{p}}\big(\big((\varphi^*)^p+\frac{t}{1-t}(\psi^*)^p\big)^{\frac{1}{p}}\big)^*(y)}\,dy \\ & =  \int_{\R^n}  (1+s)^{-\frac{n}{p}} e^{-(1+s)^{-\frac{1}{p}}\big(\big((\varphi^*)^p +s (\psi^*)^p \big)^{\frac{1}{p}}\big)^*(y) }\,dy\\ &=  \int_{\R^n}  (1+s)^{-\frac{n}{p}} e^{-(1+s)^{-\frac{1}{p}} \varphi_s(y) }\,dy, \end{align*} 
	where  $\varphi_{s}$ is given by  \eqref{varphi-t}.     Like in the proof of Lemma \ref{diffe},  let $\bar{\varphi}(x)= \lim_{s\rightarrow 0^+} \varphi_s(x)$ for every $x\in \R^n$. By Lemma \ref{boundness} and  the monotone convergence theorem,  one has $\varphi \geq \bar{\varphi}$ and $J(e^{-\bar{\varphi}})=\lim_{s\rightarrow 0^+} J(e^{-\varphi_s})\geq J(e^{-\varphi}).$  Lemma \ref{boundness} also implies that $(1+s)^{-\frac{1}{p}} \varphi_s(y)$  is decreasing for all $y\in \R^n$ with $\lim_{s\to 0^+} (1+s)^{-\frac{1}{p}} \varphi_s(y)=\bar{\varphi}(y)\leq \varphi (y).$ It follows from the  monotone convergence theorem that, $$\lim_{s\rightarrow 0^+}  \int_{\R^n}  (1+s)^{-\frac{n}{p}} e^{-(1+s)^{-\frac{1}{p}} \varphi_s(y)}  \,dy=  \lim_{s\rightarrow 0^+}  \int_{\R^n}   e^{-(1+s)^{-\frac{1}{p}} \varphi_s(y)}  \,dy =J(e^{-\bar{\varphi}}). $$

	In summary, due to $s=\frac{t}{1-t}$,  one has  \begin{align}
	\lim_{t\rightarrow 0^+}  &\frac{J((1-t)\cdot_p f\oplus_{p}t\cdot_p g)-J(f)}{t}\nonumber  \\
	&=\lim_{s\rightarrow 0^+} \frac{1+s}{s}  \int_{\R^n}   \Big((1+s)^{-\frac{n}{p}} e^{-(1+s)^{-\frac{1}{p}} \varphi_s(y)}-e^{-\varphi(y)}\Big)\,dy \nonumber \\ &=\lim_{s\rightarrow 0^+}  \int_{\R^n} \frac{ (1+s)^{-\frac{n}{p}} e^{-(1+s)^{-\frac{1}{p}} \varphi_s(y)}-e^{-\bar{\varphi}}(y)}{s} \,dy +\lim_{s\rightarrow 0^+}  \frac{J(e^{-\bar{\varphi}})-J(e^{-\varphi})}{s}.\label{chang-variable-1}
	\end{align}
	Clearly,  the second limit in \eqref{chang-variable-1} equals to $+\infty$ if $J(e^{-\bar{\varphi}})>J(e^{-\varphi})$. Note that, in this case, $\delta J_p(f, g)=+\infty$ as proved in Lemma \ref{diffe}, and this proves \eqref{decompose-1} if $J(e^{-\bar{\varphi}})>J(e^{-\varphi})$. 
	
	Now let us consider the case $J(e^{-\bar{\varphi}})=J(e^{-\varphi})$. As $\varphi \geq \bar{\varphi}$, one has  $e^{-\varphi(x)}= e^{-\bar{\varphi}(x)}$ (and hence $\varphi(x)=\bar{\varphi}(x)$) for almost all $x\in \R^n$. According to \eqref{def-entropy} and \eqref{mink1}, one sees  \begin{equation} \delta J_p(f, f)=\delta J_p(e^{-\bar{\varphi}}, e^{-\bar{\varphi}}).\label{equality-entropy} \end{equation}  Besides,  \eqref{replace-1} implies that  \begin{eqnarray} \delta J_p(f, g)=
	\lim_{t\rightarrow0^{+}}\frac{J(e^{-\varphi_t})-J(e^{-\varphi})}{t}=
	\lim_{t\rightarrow0^{+}}\frac{J(e^{-\varphi_t})-J(e^{-\bar{\varphi}})}{t}.  \label{replace-5-17}
	\end{eqnarray} Replacing $\varphi$ by $\bar{\varphi}$ in \eqref{chang-variable-1}, one has  \begin{align} 
	\lim_{s\rightarrow 0^+}  \int_{\R^n} \frac{ (1+s)^{-\frac{n}{p}} e^{-(1+s)^{-\frac{1}{p}} \varphi_s(y)}-e^{-\varphi(y)}}{s}\,dy &= \lim_{s\rightarrow 0^+}  \int_{\R^n} \frac{ (1+s)^{-\frac{n}{p}} e^{-(1+s)^{-\frac{1}{p}} \varphi_s(y)}-e^{-\bar{\varphi}(y)}}{s}\,dy \nonumber \\ &= B_1+B_2+B_3.\label{b1-b2-b3}
	\end{align} Here $B_1$ is defined and calculated as below: 
	\begin{align*}B_1&=  \lim_{s\rightarrow 0^+}  \int_{\R^n} \frac{ (1+s)^{-\frac{n}{p}} e^{-(1+s)^{-\frac{1}{p}} \varphi_s(y)}-e^{-(1+s)^{-\frac{1}{p}} \varphi_s(y)}}{s}\,dy \nonumber \\ &=  \bigg(\lim_{s\rightarrow 0^+}   \frac{ (1+s)^{-\frac{n}{p}} -1}{s} \bigg) \bigg( \lim_{s\rightarrow 0^+}   \int_{\R^n}  e^{-(1+s)^{-\frac{1}{p}} \varphi_s(y)} \,dy \bigg)   
	=  -\frac{n}{p} J\big(e^{- \bar{\varphi}}\big),  \end{align*}  where  Lemma \ref{boundness} and the monotone convergence theorem are used. It follows from \eqref{replace-5-17} that  
	\begin{align*}B_2&=  \lim_{s\rightarrow 0^+}  \int_{\R^n} \frac{e^{-\varphi_s(y)}-e^{-\bar{\varphi}(y)}}{s}\,dy 
	=\delta J_p(f, g).   \end{align*} The term $B_3$ is defined and calculated as follows:  \begin{align} B_3&=  \lim_{s\rightarrow 0^+}  \int_{\R^n} \frac{e^{-(1+s)^{-\frac{1}{p}} \varphi_s(y)}-e^{-\varphi_s(y)}}{s}\,dy=\frac{1}{p} \int_{\R^n} \bar{\varphi}(y) e^{-\bar{\varphi}(y)}\,dy. \label{estimation-B3}  \end{align} Indeed, \eqref{estimation-B3} is a consequence of the dominated convergence theorem and we now provide some details to complete the argument.  Clearly,  the integral of the second term in \eqref{estimation-B3} is actually over the domain of $\varphi_s$,  since  $(1+s)^{-\frac{1}{p}} \varphi_s(y)=+\infty$ and $\varphi_s(y)=+\infty$  if $y\notin  \dom (\varphi_s)$.  
	Note that $0\leq \varphi_s(y)<\infty$ for any $y\in \dom (\varphi_s)$. Moreover,  $1-(1+s)^{-\frac{1}{p}}$ is increasing on $s\in (0, 1)$ and $$\lim_{s\rightarrow 0^+}  \frac{1-(1+s)^{-\frac{1}{p}}}{s} =\frac{1}{p},$$ which implies the existence of a finite constant $M<\infty$ such that for all $s\in (0, 1)$,  $$0< \frac{1-(1+s)^{-\frac{1}{p}}}{s}<M\big(1-2^{-\frac{1}{p}}\big). $$ As the function $\frac{e^x-1}{x}$ is increasing on $x\in (0, \infty)$, one gets, for all $s\in (0, 1)$ and for all $y\in \dom (\varphi_s)$,  \begin{align*} 
	\frac{e^{\big(1-(1+s)^{-\frac{1}{p}} \big) \varphi_s(y)}-1}{\big(1-(1+s)^{-\frac{1}{p}} \big) \varphi_s(y)}\leq \frac{e^{\big(1-2^{-\frac{1}{p}} \big) \varphi_s(y)}-1}{\big(1-2^{-\frac{1}{p}} \big) \varphi_s(y)}. 
	\end{align*} By Lemma \ref{boundness}, for all $s\in (0, 1)$ and for any $y\in \dom (\varphi_s)$,  
	\begin{align*} \frac{e^{-(1+s)^{-\frac{1}{p}} \varphi_s(y)}-e^{-\varphi_s(y)}}{s} &=\Bigg(\frac{e^{-\varphi_s(y)}\big(1-(1+s)^{-\frac{1}{p}} \big) \varphi_s(y)}{s}\Bigg)  \Bigg( \frac{e^{\big(1-(1+s)^{-\frac{1}{p}} \big) \varphi_s(y)}-1}{\big(1-(1+s)^{-\frac{1}{p}} \big) \varphi_s(y)}\Bigg) \nonumber \\ & \leq M  e^{-\varphi_s(y)}    \Big(  {e^{\big(1-2^{-\frac{1}{p}} \big) \varphi_s(y)}-1} \Big) \nonumber \\  &\leq M  e^{ -2^{-\frac{1}{p}} \varphi_s(y)}  \leq  M  e^{ -2^{-\frac{1}{p}} \varphi_1(y)}.\end{align*} It is easily checked by  Proposition \ref{close-Lp-A-sum} that $e^{ -2^{-\frac{1}{p}} \varphi_1}\in \A_0,$ and hence $\int_{\R^n}e^{ -2^{-\frac{1}{p}} \varphi_1(y)}\,dy<+\infty.$ 
	
	For convenience, let $\bar{\varphi}(y) e^{-\bar{\varphi}(y)}=0$ if $y\notin\dom(\bar{\varphi})$. 
	By Lemma \ref{boundness},   $\dom(\bar{\varphi})\subseteq \dom(\varphi_s)\subseteq \dom(\varphi_t)$ holds for $t, s\in(0,1]$ such that $s<t$.  Thus, if there exists $s_0>0$ such that $y\notin \dom(\varphi_{s_0})$, then $y\notin \dom(\varphi_s)$ for all $s\in (0, s_0]$ and $y\notin\dom(\bar{\varphi})$; this in turn implies that $$\lim_{s\rightarrow 0^+}    \frac{e^{-(1+s)^{-\frac{1}{p}} \varphi_s(y)}-e^{-\varphi_s(y)}}{s}=0=\frac{1}{p} \bar{\varphi}(y) e^{-\bar{\varphi}(y)}.$$ On the other hand, if $y\in \dom(\varphi_s)$  for any $s\in (0, 1)$, then 
	\begin{align*} \lim_{s\rightarrow 0^+}    \frac{e^{-(1+s)^{-\frac{1}{p}} \varphi_s(y)}-e^{-\varphi_s(y)}}{s}
	&=\left(\lim_{s\rightarrow 0^+}   e^{-\varphi_s(y)}\right)  \left(\lim_{s\rightarrow 0^+}  \frac{e^{\big(1-(1+s)^{-\frac{1}{p}}\big) \varphi_s(y)}-1}{s} \right) 	=\frac{1}{p}\bar{\varphi}(y) e^{-\bar{\varphi}(y)}.\end{align*}  Hence, the dominated convergence theorem can be applied to $B_3$ and get \eqref{estimation-B3}.  
	
	Summing up $B_1, B_2$ and $B_3$, by  \eqref{chang-variable-1},  \eqref{equality-entropy}, \eqref{b1-b2-b3} and Lemma \ref{specialcaseoffg}, one gets \begin{align*}
	\lim_{t\rightarrow 0^+}  \frac{J((1-t)\cdot_p f\oplus_{p}t\cdot_p g)-J(f)}{t} 
	&=B_1+B_2+B_3=\delta J_{p}(f,g)-\delta J_{p}(f,f).
	\end{align*} Hence, formula \eqref{decompose-1}  is obtained. This completes the proof of Lemma \ref{lemma35}.  \end{proof} 

We are now ready to prove our Theorem \ref{minkowskiinequality}. 
\begin{proof}[Proof of Theorem \ref{minkowskiinequality}]  Let $f=e^{-\varphi}\in \A_0$ and $g=e^{-\psi}\in\A_0$ such that $J(f)>0$ and $J(g)>0$. For $p>1$,  let $F(t)=\log J((1-t)\cdot_p f\oplus_{p} t\cdot_p g)$ for $t\in (0, 1)$. Let $$F(0+)=\lim_{t\to 0^+}F(t) \ \ \mathrm{and}\ \ F(1-)=\lim_{t\to 1^-}F(t).$$  Note that $F(0+)\geq \log J(f)$ and $F(1-)\geq \log J(g)$.  According to Corollary \ref{concave-PL-1}, $F$ is a concave function on $t\in (0, 1)$ and thus $$F(t) \geq (1-t)F(0+)+tF(1-)\geq \log J(f)+t(\log J(g)-\log J(f))$$ holds for $t\in (0, 1)$.   Consequently, if $F(0+)>\log J(f)$, then 
$$  \lim_{t\rightarrow0^+} \frac{J((1-t)\cdot_p f\oplus_{p} t\cdot_p g)-J(f)}{t} =+\infty > J(f)\log\Big(\frac{J(g)}{J(f)}\Big).$$ While if $F(0+)=\log J(f)$, then  \begin{eqnarray}   \lim_{t\rightarrow0^+} \frac{J((1-t)\cdot_p f\oplus_{p} t\cdot_p g)-J(f)}{t} =\frac{\,d e^{F(t)}}{\,dt}\bigg|_{t=0^+}\!\!  =e^{F(t)} \frac{\,d F(t)}{\,dt}\bigg|_{t=0^+}\geq J(f)\log\Big(\frac{J(g)}{J(f)}\Big).\label{derivative--1} \end{eqnarray}
	Lemmas   \ref{specialcaseoffg} and  \ref{lemma35} then  yield the desired inequality \eqref{mink1} as follows:
	\begin{align}  \delta J_{p}(f,g)&= \delta J_{p}(f,f) + \lim_{t\rightarrow0^+} \frac{J((1-t)\cdot_p f\oplus_{p} t\cdot_p g)-J(f)}{t} \nonumber \\ & \geq  \delta J_{p}(f,f)+J(f)\log\Big(\frac{J(g)}{J(f)}\Big) \label{ine-with-d-ff}  \\  
	&=  J(f)\left[ \frac{n}{p} +\frac{1-p}{p}\log J(f)+\log J(g)\right]+\frac{1}{p}{\rm Ent}(f). \nonumber\end{align}  
	
	Now let us characterize the equality for \eqref{mink1}. It is obvious that \eqref{mink1} becomes equality if $f=g$ by  Lemma \ref{specialcaseoffg}.  Conversely, assume that (\ref{mink1}) holds with equality sign which happens only in the case $F(0+)=\log J(f)$; it requires equality in \eqref{derivative--1}. In particular, as $F(0+)= \log J(f)$, one has 
	$$ F'(0^{+})=\frac{\,d F(t)}{\,dt}\bigg|_{t=0^+} = \log\Big(\frac{J(g)}{J(f)}\Big).$$ Note that $F(t)\leq F(0+)+t F'(0+)$ for all $t\in (0, 1)$ since $F$ is concave on $(0, 1)$. This gives $$F(t) \leq \log J(f)+t(\log J(g)-\log J(f)).$$ Consequently, equality holds in the Pr\'ekopa-Leindler type
	inequality (\ref{prekopainequality1}) and then $f=g$.
\end{proof}

The following corollary provides  a unique determination of log-concave functions. 	
\begin{corollary}\label{cormink-uniq}  Let $f_1, f_2 \in\A_0$ such that $J(f_1)=J(f_2)>0$. For $p>1$, if \begin{equation}\label{incroci-uniq}
	\delta J_p (f_1, g) = \delta J_p (f_2, g) \end{equation} holds for any $g\in \A_0$ with $J(g)>0$, then $f_1=f_2$.  \end{corollary}
\begin{proof} By letting $g=f_1$ in \eqref{incroci-uniq}, it follows from Theorem \ref{minkowskiinequality}
	(or \eqref{ine-with-d-ff})   and $J(f_1)=J(f_2)$  that \begin{eqnarray}\label{incroci-uniq-11}
	\delta J_p (f_1, f_1) = \delta J_p (f_2, f_1) \geq \delta J_{p}(f_2, f_2)+J(f_2)\log\Big(\frac{J(f_1)}{J(f_2)}\Big) 
	=\delta J_{p}(f_2, f_2)
	\end{eqnarray} with equality if and only if $f_1=f_2$. Similarly, \begin{equation}\label{incroci-uniq-22}
	\delta J_p (f_2, f_2) = \delta J_p (f_1, f_2) \geq \delta J_{p}(f_1, f_1)+J(f_1)\log\Big(\frac{J(f_2)}{J(f_1)}\Big)=\delta J_{p}(f_1, f_1).\end{equation}   This means that  \eqref{incroci-uniq-11} holds with equality, which in turn gives $f_1=f_2$ as desired.
\end{proof}

\section{An explicit formula for $\delta J_p(f, g)$ }   \label{section-varition-516} 	\setcounter{equation}{0} 
Our goal in this section is to obtain an explicit integral formula for $\delta J_p(f, g)$ for $p>1$ under additional conditions on $f$ and $g$. Again the case $p=1$ has been discussed in \cite{CF13} and hence will not be covered in this section. We shall need the subclass  $\A'_0\subset \A_0$ where $\A_0'=\{f=e^{-\varphi}: \varphi \in \L'_0\}$ 
with $\L_0'\subset\L_0$ given by  
\[\L'_0:=\Big\{\varphi\in\L_0: \varphi \in \C^1(\R^n)\cap\C^2_+(\R^n\setminus\{o\}) \ \mathrm{is\ strictly\ convex\ and \ supercoercive\ with}\ \dom(\varphi)=\R^n\Big\}. 
\]  Hereafter, a function $\varphi$ is called supercoercive if  $\lim_{|x|\rightarrow\infty} \frac{\varphi(x)}{|x|}=+\infty.$  
It is easily checked that $(\R^n, \varphi)$ for $\varphi\in \L'_0$ is a pair satisfying that $\varphi$ is differentiable and strictly convex on $\R^n$, and  \begin{equation} \label{ess-smooth} \lim_{i\to\infty} |\nabla \varphi(x_i)|\rightarrow+\infty\ \ \mathrm{for\ each\ sequence }\  \ \{x_i\}_{i\in\mathbb{N}}\subset \R^n \ \ \mathrm{such\ that}\  \lim_{i\to \infty} |x_i|=+\infty.  \end{equation}   This pair usually is called {\it a convex function of Legendre type} (see e.g. \cite[Section 26]{Roc70} for more general definitions and properties). In general, \eqref{ess-smooth} holds automatically for $\varphi$  if $\dom(\varphi^*)=\R^n$ and $\varphi$ is a differentiable convex function with $\dom(\varphi)=\R^n$, due to  \cite[Lemma 26.7]{Roc70}. 

We say that $(D, \psi)$  is  the \emph{Legendre conjugate} of $(C, \varphi)$ if
$$
\psi(y)=\langle x,y\rangle-\varphi(x),\quad \mathrm{for\ any}\  y\in D\ \mathrm{and \ for \ any }\ x \in \nabla \varphi^{-1}(y)=\{z\in C: \ y= \nabla \varphi(z)\}, 
$$ where $D=\{y\in \R^n: y=\nabla \varphi(x), \  x\in C\}$. Theorem 26.5 in  \cite{Roc70}  provides a nice result regarding the relation between the Legendre conjugate and Fenchel conjugate. We shall not need the full statement of  \cite[Theorem 26.5]{Roc70}, and only the special cases, when both domains are $\R^n$, will be stated in the following lemma. 

\begin{lemma}\label{proplegendre} Let $\phi\in \C^1(\R^n)$ be  such that $ \dom(\phi^*)=\R^n$.  Then $(\R^n, \phi)$ is a convex function of Legendre type if and only if $(\R^n, \phi^*)$ is. When these conditions hold, $(\R^n, \phi ^*)$ is the Legendre conjugate of $(\R^n, \phi)$ (and vice verse). Moreover, both $\nabla \phi: \R^n \to \R^n$  and $\nabla \phi^*:\R^n\to \R^n$ are continuous bijections and satisfy that $\nabla \phi ^*$ is the inverse of $\nabla \phi$ (namely,  $\nabla ^{-1}  \phi= \nabla \phi ^*$).\end{lemma}

We now prove some lemmas before we state our main result in this section.  

\begin{lemma}\label{vainshat0}
	If  $\varphi\in \L'_0$, then $\varphi^*\in \L_0'$. Moreover,  $\nabla \varphi(o)=o$, $\nabla\varphi^*(o)=o$,  and 
	$$\{x\in\R^n: \varphi(x)=0\}=\{x\in\R^n: \varphi^*(x)=0\}=\{o\}.$$ 
\end{lemma}
\begin{proof} It is clear that $\varphi^{**}=\varphi$,  as   $\varphi$ is convex and $\varphi\in \C^1(\R^n)$.  According to the Moreau-Rockafellar theorem (see e.g. \cite[Proposition 3.5.4]{BV10}),  a proper  lower semi-continuous convex function $\varphi$ on $\R^n$ is supercoercive if and only if $\dom(\varphi^*)=\R^n$. Consequently, $\dom(\varphi^*)=\R^n$ and  $\varphi^*$ is  supercoercive, due to the facts that $\varphi$ is   supercoercive, and respectively $\dom(\varphi)=\R^n$. It is also trivial to have $\varphi^*(o)=0$ and $\varphi^*(y)\geq 0$ for all $y\in \R^n$, as $\varphi\in\L_0$. Note that  the pair $(\R^n, \varphi)$ is a convex function of Legendre type, so is  the pair $(\R^n, \varphi^*)$ by Lemma \ref{proplegendre}. In particular, $\varphi^*$ is strictly convex on $\R^n$.  Lemma \ref{proplegendre} also implies that  $\nabla \varphi$ and its inverse $\nabla \varphi^*$ are both continuous on $\R^n$ and thus $\varphi^*\in \C^1(\R^n)$.  As $\varphi\in \C^2_+(\R^n\setminus\{o\})$, the Hessian matrix $\nabla^2\varphi$ is positive definite and continuous on  $\R^n\setminus\{o\}$. It follows from  the inverse mapping theorem   that $\varphi^*\in \C^2_+(\R^n\setminus \{o\})$, which  concludes  $\varphi^*\in \L_0'$.   In particular, $\varphi(x)+\varphi^*(y)=\langle x, y\rangle$ holds for all $x, y\in \R^n$ such that $y=\nabla \varphi(x)$ (and $x=\nabla\varphi^*(y)$).  The strict convexity of $\varphi$ implies that $\varphi$ has a unique minimizer. As $\varphi\in \L_0$, $\varphi$ attains its minimum at $o$, hence $\nabla\varphi(o)=o$ and  $\{x\in\R^n: \varphi(x)=0\}=\{o\}$. The same arguments clearly work for $\varphi^*$, and this concludes the proof of Lemma \ref{vainshat0}. \end{proof}

\begin{lemma}\label{belongtothesameclass}
	Let $\varphi, \psi \in \L'_0$. For $p>1$ and   $t>0$, set  $\varphi_{t}:=\varphi\square_{p}( \psi\cdot_pt).$ Then $\varphi_{t}\in \L_0'$. Moreover, both $(\R^n, \varphi_t)$ and $(\R^n, \varphi_t^*)$ are convex functions of Legendre type. 
\end{lemma}

\begin{proof} Let $\varphi, \psi \in \L'_0\subset \L_0$ be two convex functions.  According to Proposition \ref{close-Lp-A-sum},  $\varphi_t\in \L_0$ for all $t>0$. In particular, $\varphi_t$ is non-negative on $\R^n$. By Lemma \ref{boundness}, $0\leq \varphi_t\leq \varphi$ for all $t>0$ and $p>1$. This implies $\dom (\varphi_t)=\R^n$. This, together with the Moreau-Rockafellar theorem (see e.g. \cite[Proposition 3.5.4]{BV10}), immediately implies that $\varphi_t^*$ is supercoercive. On the other hand, by Proposition \ref{close-Lp-A-sum},  for any $t>0$,  $\dom(\varphi_t^*)$ is clearly equal to $\R^n$ and thus $\varphi_t$ is supercoercive.

Let us check the differentiability of $\varphi_t$ and $\varphi_t^*$. As explained in the proof of Lemma \ref{vainshat0},  if $\varphi\in \L'_0$,  then 
$\nabla \varphi$ and its inverse $\nabla \varphi^*$ are both continuous on $\R^n$ and continuously differentiable on $\R^n\setminus\{o\}$. Moreover,   $\nabla^2\varphi$ and   $\nabla ^2\varphi^*$ are positive definite and continuous on  $\R^n\setminus\{o\}$. 	  Similar properties hold for $\psi$.    Let $x\neq o$. In this case, both  $\varphi^*(x)$ and $\psi^*(x)$ are strictly positive due to Lemma \ref{vainshat0}.  It follows from  \eqref{varphi-t-star}  that,   for all $t>0$ and $p>1$,   $\varphi_t^*(x)>0$ and  \begin{equation}\label{first-order-deriv}\nabla\varphi_t^*(x)=\frac{ \big(\varphi^*(x)\big)^{p-1} \nabla\varphi^*(x)+t(\psi^*(x))^{p-1} \nabla\psi^*(x)}{ \big(\varphi^*_t(x)\big)^{p-1} }. \end{equation}  Clearly  $\nabla\varphi_t^*(x)$ is continuous at $o\neq x\in \R^n$.  
	On the other hand, one can verify $\nabla\varphi_t^*(o)=o$ according to the usual definition of differentiability:
	\begin{align*} \lim_{z\to o} \frac{\varphi_t^*(z)-\varphi_t^*(o)-\langle o, z-o\rangle}{|z-o|}&=\lim_{z\to o} \frac{\big((\varphi^*(z))^p+t(\psi^*(z))^p\big)^{\frac{1}{p}}}{|z|}\\ &=\lim_{z\to o} \bigg(\Big(\frac{\varphi^*(z)}{|z|}\Big)^p+t\Big(\frac{\psi^*(z)}{|z|}\Big)^p\bigg)^{\frac{1}{p}}=0,\end{align*} where the last equality follows from  $\nabla \varphi^*(o)=\nabla\psi^*(o)=o$ due to 
	Lemma \ref{vainshat0}.
	Moreover, 
	$$\lim_{x\rightarrow o}\nabla \varphi^*(x)=\nabla \varphi^*(o)=o\ \ \mathrm{and}\ \  \lim_{x\rightarrow o}\nabla \psi^*(x)=\nabla \psi^*(o)=o.$$ These conclude that,  for $p>1$,  $\varphi_t^*$ is continuously differentiable on $\R^n$, because  \begin{align*} \lim_{x\to o} \big|\nabla\varphi_t^*(x)\big|&=\lim_{x\to o} \bigg|\frac{ \big(\varphi^*(x)\big)^{p-1} \nabla\varphi^*(x)+t(\psi^*(x))^{p-1} \nabla\psi^*(x)}{ \big(\varphi^*_t(x)\big)^{p-1} }\bigg|\\ &\leq \lim_{x\to o} \big|\nabla\varphi^*(x)\big|   \bigg(\frac{\varphi^*(x)}{\varphi^*_t(x)}\bigg)^{p-1} +t^{\frac{1}{p}}  \lim_{x\to o} \big|\nabla\psi^*(x)\big|   \bigg(\frac{t (\psi^*)^p(x)}{(\varphi^*_t)^p(x)}\bigg)^{1-\frac{1}{p}} \\ &\leq \lim_{x\to o} \big|\nabla\varphi^*(x)\big|   +t^{\frac{1}{p}}  \lim_{x\to o} \big|\nabla\psi^*(x)\big|   =0.\end{align*}

Now let us check that $(\R^n, \varphi_t^*)$ is a convex function of Legendre type. To this end, as $\varphi_t^*$ is already proved to be differentiable,   we only need to verify that $\varphi^*_t$ is  strictly convex on $\R^n$ and  \eqref{ess-smooth} holds for $\varphi^*_t$. As mentioned before,  \eqref{ess-smooth} for $\varphi^*_t$ holds automatically because $\dom(\varphi)=\R^n$ and $\varphi_t^*$ is differentiable, due to  \cite[Lemma 26.7]{Roc70}. The strictly convexity of $\varphi_t^*$ is easily checked as follows:  for $p>1$, $\lambda\in (0, 1)$ and $x, y\in \R^n$ such that $x\neq y$,  by the Minkowski inequality for norms,  \eqref{varphi-t-star} and Lemma \ref{vainshat0} (which implies the strict convexity of $\varphi^*$ and $\psi^*$), one has \begin{align*} 
	\varphi_t^*(\lambda x+(1-\lambda)y)& =\Big(\big(\varphi^*(\lambda x+(1-\lambda)y)\big)^p+t\big(\psi^*(\lambda x+(1-\lambda)y)\big)^p\Big)^{\frac{1}{p}}\\ &<\Big(\big(\lambda \varphi^*(x)+(1-\lambda)\varphi^*(y)\big)^p+t\big(\lambda \psi^*(x)+(1-\lambda)\psi^*(y)\big)^p\Big)^{\frac{1}{p}}\\&\leq \lambda \varphi_t^*(x)+(1-\lambda)\varphi_t^*(y).
	\end{align*} Therefore,  $(\R^n, \varphi_t^*)$ is  a convex function of Legendre type, and so is  $(\R^n, \varphi_t)$  due to Lemma \ref{proplegendre}. Moreover, both $\nabla \varphi_t^*$ and its inverse $\nabla \varphi_t$ are continuous. Thus $\varphi_t\in \C^1(\R^n)$ and is strictly convex.

According to Lemma \ref{vainshat0},  $\nabla^2\varphi^*$ and $\nabla^2\psi^*$ are both positive definite and continuous on  $\R^n\setminus\{o\}$. Let $x\in \R^n\setminus\{o\}$ and by \eqref{first-order-deriv}, one has \begin{align} \nabla^2\varphi_t^* &=(p-1)\frac{(\varphi^*)^{p-2}\nabla\varphi^*\otimes\nabla\varphi^*+t(\psi^*)^{p-2}\nabla\psi^*\otimes\nabla\psi^*-(\varphi_t^*)^{p-2}\nabla\varphi_t^*\otimes\nabla\varphi_t^*}{(\varphi_t^*)^{p-1}}\nonumber \\ &\quad \quad + \frac{(\varphi^*)^{p-1}\nabla^2\varphi^*+t(\psi^*)^{p-1}\nabla^2\psi^*}{(\varphi_t^*)^{p-1}},\label{second-hession} \end{align} where $y\otimes y$ is the rank $1$ matrix generated by $y\in \R^n$.   Clearly $\nabla^2\varphi_t^*$ is continuous on $\R^n\setminus \{o\}$ and positive definite whose determinant is strictly positive (the calculation is standard and hence will be omitted). So $\varphi_t^*\in \C^2_+(\R^n\setminus \{o\})$ for all $t>0$ and $p>1$. Together with $(\nabla \varphi_t) ^ {-1} = \nabla \varphi_t ^*$, the inverse mapping theorem gives  that $\varphi_t\in \C^2_+(\R^n\setminus \{o\})$ for all $t>0$ and $p>1$. Moreover, for any $o\neq x\in \R^n$, $\nabla^2\varphi_t(x) $ equal the inverse of $\nabla^2 \varphi_t^*(y)$ with $y=\nabla \varphi_t(x)$. This concludes that $\varphi_t\in \L'_0$.  \end{proof}

We will also need the following lemma. The function $\varphi_t$ is as in \eqref{varphi-t}.  

\begin{lemma}\label{continuityforgradient}
	Let $\varphi, \psi \in \L'_0$. For $p>1$ and  $t>0$,  one has,\\ 
	\noindent i)  $\lim_{t\rightarrow0^{+}}\varphi_{t}(x)=\varphi(x)$ for all $x\in \R^n$;\\
	\noindent ii) for every closed bounded subset $E\subset \R^n$, $\lim_{t\rightarrow0^{+}}\nabla \varphi_{t}(x)=\nabla \varphi(x)$ uniformly on $E$.
\end{lemma}

\begin{proof}	
	i) Let $\varphi, \psi \in \L'_0$.  By Lemma \ref{boundness}, for any $x\in \R^n$, $\varphi_t(x)$ is decreasing on $t\in (0, 1]$, and thus $
	\limsup_{t\rightarrow0^{+}}\varphi_{t}(x)\leq \varphi (x). $  The desired argument in i) follows immediately once the following is checked: for any $x\in \R^n$, \begin{equation}\label{inflimit}
	\liminf_{t\rightarrow 0^+} \varphi_{t}(x)\geq \varphi(x).
	\end{equation} To this end, let  $x\in \R^n$ be fixed and $r>|\nabla \varphi(x)|$.  Let $B_r$ be the Euclidean ball with center at the origin and radius $r$. For $p>1$, it can be checked, by  $\varphi^*, \psi^*\geq 0$,  that  for all $t>0$, \begin{equation}\label{compare-5-25}
	\varphi_t^*=\big((\varphi^*)^p+t(\psi^*)^p\big)^{\frac{1}{p}}\leq\varphi^*+ t^{\frac{1}{p}} \psi^*.
	\end{equation}
	It follows from 
	\eqref{compare-5-25}  that, for $t\in (0, 1]$,  \begin{align*}
	\varphi_{t}(x) =\underset{y\in\mathbb{R}^n}{\sup}\Big\{\langle x,y\rangle -\varphi_{t}^*(y)\Big\} \geq \underset{y\in B_r}{\sup}\Big\{\langle x,y\rangle -\varphi_{t}^*(y)\Big\}\geq\underset{y\in B_r}{\sup}\Big\{\langle x,y\rangle-\varphi^*(y)-t^{\frac{1}{p}}\psi^*(y)\Big\}. \end{align*}	 Define the finite constant $c$ to be $c=\max\{\psi^*(y): y\in B_{r}\}$. The fact that $r>|\nabla \varphi(x)|$ implies $\nabla \varphi(x)\in B_r$, and hence, for $t\in (0, 1]$, 
	\begin{align*}
	\varphi_{t}(x)
	\geq\underset{y\in B_r}{\sup}\Big\{\langle x,y\rangle-\varphi^*(y)-t^{\frac{1}{p}}\psi^*(y) \Big\}
	\geq \langle x,\nabla \varphi(x)\rangle-\varphi^*(\nabla \varphi(x))-t^{\frac{1}{p}}c=\varphi(x)-t^{\frac{1}{p}}c.
	\end{align*} 	
	The desired inequality (\ref{inflimit}) follows by letting $t\rightarrow 0^+$. This completes the proof for part i). 	 
	
	\vskip 2mm \noindent ii) This is a direct consequence of \cite[Theorem 25.7]{Roc70}; in a slight different form, it reads: if $\{\phi_i\}_{i\in \mathbb{N}\cup \{0\}}$ is a sequence of finite and differentiable convex functions on an open convex set $E$ such that $\phi_i\rightarrow \phi_0$ pointwisely on $E$, then $\nabla \phi_i\rightarrow \nabla \phi$ pointwisely on $E$ and uniformly on every closed bounded subset of $E$.  \end{proof}

The following lemma provides the derivative of $\varphi_t$ with respect to $t$. 

\begin{lemma}\label{variationorigin} Let $\varphi, \psi \in \L'_0$. For $p>1$ and  for $t>0$,  set $\varphi_{t}=\varphi \square_{p}( \psi  \cdot_p t)$.  Then for all $x\in \R^n$ and $t>0$, one has 
	\begin{equation}\label{derivative-varphi-5-27--1} \frac{d}{dt}\varphi_{t}(x)=-\frac{1}{p}\Big(\psi^*(\nabla \varphi_{t}(x))\Big)^p \Big(\varphi^*_{t}(\nabla \varphi_{t}(x))\Big)^{1-p}.\end{equation}  
	In particular, for  all $x\in \R^n$, one has 
	\begin{equation}\label{derivative-varphi-5-27} 
	\frac{d}{dt}\varphi_{t}(x)\bigg|_{t=0^+}=-\frac{1}{p}\Big(\psi^*(\nabla \varphi(x))\Big)^p \Big(\varphi^*(\nabla \varphi(x))\Big)^{1-p}. \end{equation} 
\end{lemma}

\begin{proof} Let $\varphi, \psi \in \L'_0$ and $p>1$. Formulas  \eqref{derivative-varphi-5-27--1} and  \eqref{derivative-varphi-5-27}  clearly hold for $x=o$ following from Lemmas \ref{vainshat0} and \ref{belongtothesameclass} (the latter one gives $\varphi_t\in \L_0$ and hence $\nabla \varphi_t(o)=o$). 
	
	Let $x\neq o$. By Lemma \ref{belongtothesameclass} (and its proof), one sees  that 
	the mapping $F$ defined by $
	F(x, y, t) = \nabla \varphi_t^*(y)-x$
	is continuously differentiable  on $(\R ^n\setminus\{o\}) \times (\R ^n\setminus\{o\}) \times (0, + \infty)$.  Note that $\frac{\partial F}{\partial y}=\nabla^2\varphi_t^*(y)$ is nonsingular for every $y \in \R ^n\setminus\{o\}$ by \eqref{second-hession}. The implicit function theorem yields (locally) the existence of a unique continuously differentiable mapping $y = y (x, t)$ for $(x, t)\in (\R^n\setminus\{o\})\times (0, \infty)$ such that $F(x, y(x, t), t)=o.$ That is, $x=\nabla\varphi_t^*(y(x, t)).$  According to Lemma \ref{belongtothesameclass}, $\nabla \varphi_{t}=\nabla^{-1} \varphi^*_{t}$. Thus,  $y(x, t)=\nabla \varphi_t(x)$ and $x= \nabla\varphi_t^*\big(\nabla \varphi_t(x)\big)$ for $x\neq o$. Moreover,  for $x\neq o$, one has $\varphi_{t}(x)=\langle x,\nabla \varphi_{t}(x)\rangle-\varphi^*_{t}(\nabla \varphi_{t}(x))$.  Taking  derivative from both sides, one gets, for any  $t\in (0, \infty)$ and (fixed) $x\in \R^n\setminus\{o\}$, 
	\begin{align*}
	\frac{d}{dt}\varphi_{t}(x)&=\langle x, \frac{d}{dt}\nabla \varphi_{t}(x)\rangle-\frac{1}{p}\Big(\psi^*(\nabla \varphi_{t}(x))\Big)^p \Big(\varphi^*_{t}(\nabla \varphi_{t}(x))\Big)^{1-p}-\langle\nabla\varphi^*_{t}(\nabla \varphi_{t}(x)), \frac{d}{dt}\nabla \varphi_{t}(x)\rangle\\
	&=-\frac{1}{p}\Big(\psi^*(\nabla \varphi_{t}(x))\Big)^p \Big(\varphi^*_{t}(\nabla \varphi_{t}(x))\Big)^{1-p}.
	\end{align*}	 This concludes the proof of \eqref{derivative-varphi-5-27--1}. Consequently, \eqref{derivative-varphi-5-27} follows from part ii) of Lemma \ref{continuityforgradient} and by letting $t\rightarrow 0^+$ in \eqref{derivative-varphi-5-27--1}. 
\end{proof}	

In order to obtain an explicit formula for $\delta J_p(f, g)$, we need to define the notion of admissible $p$-perturbation.  See \cite{CF13} for the case for $p=1$. 

\begin{definition}\label{p-perturbation} Let $f=e^{-\varphi}\in \A_0$ and $p>1$. The function $g=e^{-\psi}\in \A_0$ is said to be an admissible $p$-perturbation for $f$, if there exists a constant $c>0$ such that $(\varphi^*)^p-c (\psi^*)^p$ is a convex function.  \end{definition}

Our main result in this section is the following integral formula for  $\delta J_p(f, g)$. Again, we only focus on $p>1$ and the case $p=1$ has been covered in \cite[Theorem 4.5]{CF13}.

\begin{theorem}\label{variationformula}
	Let $f=e^{-\varphi}\in \A'_0$ and $g=e^{-\psi}\in \A'_0$.  For $p>1$, assume that $g$ is an admissible $p$-perturbation for $f$. In addition, suppose that there exists a constant $k>0$ such that 
	\begin{equation} \label{compatible-1} \det\Big(\nabla^2 (\varphi^*)^p (y) \Big)\leq k \big(\varphi^*(y)\big)^{n(p-1)}  \det\big(\nabla^2 \varphi^* (y) \big)
	\end{equation} holds for all $y\in \R^n\setminus\{o\}$.  Then  		\begin{align}
	\delta J_{p}(f,g)&=\frac{1}{p}\int_{\mathbb{R}^n}(\psi^*(\nabla \varphi(x)))^{p}(\varphi^*(\nabla\varphi(x)))^{1-p}e^{-\varphi(x)}\,dx \nonumber \\&=\frac{1}{p}\int_{\mathbb{R}^n}(\psi^*(y))^{p}(\varphi^*(y))^{1-p}\,d\mu(f, y).\label{tesi'}\end{align} \end{theorem}
\begin{proof} Let $f=e^{-\varphi}\in \A'_0$ and $g=e^{-\psi}\in \A'_0$.  Let $t>0$ be fixed. According to \eqref{varphi-t}, Proposition \ref{close-Lp-A-sum}, and Definition \ref{variation-p-5-15-1}, one sees that   		\begin{equation}\label{difference-5-28}
	\delta J_{p}(f,g)=\lim_{t\rightarrow0^+} \frac{J(f\oplus_{p} t\cdot_p g)-J(f)}{t}=\lim_{t\rightarrow0^+} \int_{\R^n} \frac{e^{-\varphi_t(x)}-e^{-\varphi(x)}}{t}\,dx. \end{equation}  By Lemma \ref{variationorigin} (in particular, \eqref{derivative-varphi-5-27}),  it holds that, for $x\in \R^n$, \begin{equation}\label{difference-5-28-1}
	\lim _{t \rightarrow 0^+} \frac{e^{-\varphi_{t}(x)}-e^{-\varphi(x)}}{t}=\frac{1}{p}\Big(\psi^*(\nabla \varphi(x))\Big)^p \Big(\varphi^*(\nabla \varphi(x))\Big)^{1-p} e^{-\varphi(x)}.\end{equation} Consequently, the first formula in \eqref{tesi'} follows immediately from \eqref{difference-5-28} and \eqref{difference-5-28-1} once the dominated convergence theorem is verified. The second formula in \eqref{tesi'} follows directly from Definition \ref{def-moment-measure}. 
	
	Now let us verify that the 	dominated convergence theorem can be applied for 	\eqref{difference-5-28}. For any fixed $x\in \R^n$, Lemma \ref{continuityforgradient} implies that $\varphi_t(x)$ is continuous at $t=0$ and Lemma \ref{variationorigin} yields the differentiability of $\varphi_t(x)$ on $t\in (0, \infty)$. Together with \eqref{derivative-varphi-5-27--1}, the Lagrange mean value theorem can be applied to the function $t \mapsto e^{-\varphi_t(x)}$ and obtain that  there exists an $s\in (0, t)$, such that  \begin{equation}\label{Lagrange-5-28}\frac{e^{-\varphi_{t}(x)}-e^{-\varphi(x)}}{t-0}=\frac{d}{dt}e^{-\varphi_{t}(x)} \bigg|_{t=s}=\frac{1}{p}\Big(\psi^*(\nabla \varphi_{s}(x))\Big)^p \Big(\varphi^*_{s}(\nabla \varphi_{s}(x))\Big)^{1-p}e^{-\varphi_s(x)}.\end{equation}  The function on the right, for any $s\in [0, t]$, is integrable over $\R^n$, namely, 
	\begin{align} \Psi(s)=\frac{1}{p} 
	\int_{\R^n} \Big(\psi^*(\nabla \varphi_{s}(x))\Big)^p \Big(\varphi^*_{s}(\nabla \varphi_{s}(x))\Big)^{1-p}e^{-\varphi_s(x)}\,dx<\infty \label{bounded-5-28-1}. \end{align} Indeed, \eqref{varphi-t-star}  yields $s(\psi^*)^p\leq (\varphi_s^*)^p$ for any $s>0$. Together with  \eqref{finitezza-1}, one has   \begin{align*}
	s \Psi(s) &= \frac{1}{p} \int_{\R^n} s\Big(\psi^*(\nabla \varphi_{s}(x))\Big)^p \Big(\varphi^*_{s}(\nabla \varphi_{s}(x))\Big)^{1-p}e^{-\varphi_s(x)}\,dx \\
	&\leq \frac{1}{p}  \int_{\mathbb{R}^{n}}  \Big(\varphi^*_{s}(\nabla \varphi_{s}(x))\Big)^{p} \Big(\varphi^*_{s}(\nabla \varphi_{s}(x))\Big)^{1-p}e^{-\varphi_s(x)}\,dx\\
	&= \frac{1}{p}  \int_{\mathbb{R}^{n}}   \varphi^*_{s}(\nabla \varphi_{s}(x))e^{-\varphi_s(x)}\,dx<+\infty.
	\end{align*}  For $s=0$, the assumption that $g$ is an admissible $p$-perturbation of $f$ is needed.  Recall that $\varphi^*(o)=\psi^*(o)=0$ if $f=e^{-\varphi}\in \A_0$  and $g=e^{-\psi}\in \A_0$. According to Definition \ref{p-perturbation}, there exists a constant $c>0$, such that $(\varphi^*)^p-c (\psi^*)^p$ is a convex function. By  Lemma \ref{vainshat0}, one has  $(\varphi^*(o))^p-c (\psi^*(o))^p=0$ and $ \nabla\big((\varphi^*)^p-c(\psi^*)^p\big)(o)=0$, which in turn yields that, for any $y\in \R^n$, 
	\begin{eqnarray}\label{p-add-compare-1}
	\big((\varphi^*)^p-c(\psi^*)^p\big)(y)\geq \big((\varphi^*)^p-c(\psi^*)^p\big)(o)- \langle y, \nabla\big((\varphi^*)^p-c(\psi^*)^p\big)(o) \rangle =0. 
	\end{eqnarray}	 Consequently, by \eqref{finitezza-1},  the following holds:
	\begin{align*} 
	c \Psi(0)=\int_{\R^n} c\big(\psi^*(\nabla \varphi (x))\big)^p  \big(\varphi^*(\nabla \varphi(x))\big)^{1-p} e^{-\varphi(x)}\,d x \leq  \int_{\R^n}  \varphi^*(\nabla \varphi(x))  e^{-\varphi(x)}\,d x <+\infty. \end{align*}
	By \eqref{difference-5-28},  \eqref{Lagrange-5-28} and \eqref{bounded-5-28-1}, we obtain that, for all $s\in (0, t)$, 
	\begin{align}\delta J_{p}(f,g) 
	= \frac{1}{p}\lim_{s\to 0^+}  \int_{\R^n}  \big(\psi^*(\nabla \varphi_{s}(x))\big)^p \big(\varphi^*_{s}(\nabla \varphi_{s}(x))\big)^{1-p}e^{-\varphi_s(x)}\,dx.  \label{Lagrange-5-28-22}\end{align} Let $y=\nabla \varphi_{s}(x)$. From Lemmas \ref{vainshat0} and \ref{belongtothesameclass} (in particular, its proof),  \eqref{Lagrange-5-28-22} can be rewritten as 
	\begin{align}  
	\delta J_{p}(f,g)     &= \frac{1}{p}\lim_{s\to 0^+}  \int_{\R^n\setminus\{o\}}  \big(\psi^*(\nabla \varphi_{s}(x))\big)^p \Big(\varphi^*_{s}(\nabla \varphi_{s}(x))\Big)^{1-p}e^{-\varphi_s(x)}\,dx \nonumber \\ &= \frac{1}{p}\lim_{s\to 0^+}  \int_{\R^n\setminus\{o\}}  \big(\psi^*(y)\big)^p \Big(\varphi^*_{s}(y)\Big)^{1-p}e^{-\varphi_s(\nabla\varphi_s^*(y))} \det \big(\nabla^2\varphi^*_s(y)\big) \,dy. \label{Lagrange-5-31-22}
	\end{align} The desired formula \eqref{tesi'} follows once the dominated convergence theorem is verified for \eqref{Lagrange-5-31-22}. 
	
	According to Lemma \ref{belongtothesameclass}, one has $\varphi^*_t\in \C^2_+(\R^n\setminus\{o\})$.    For $y\neq o$,  \eqref{second-hession}  can be rewritten as \begin{align} \nabla^2\big((\varphi_t^*)^p\big)(y)&=\nabla^2\big((\varphi^*)^p\big)(y)+t\nabla^2\big((\psi^*)^p\big)(y)\nonumber  \\&= \Big(p(p-1)\big(\varphi_t^*\big)^{p-2} \nabla\varphi^*_t\otimes\nabla\varphi^*_t+p\big(\varphi_t^*\big)^{p-1}\nabla^2\varphi^*_t\Big)(y). \label{second-hession-5-30} \end{align}  For a positive definite matrix $A$  and  a semi-definite matrix $B$, the following holds:  \begin{equation} \label{matrix-ineq} \det (A+B)\geq \det A+\det B,\end{equation}  where $\det A$ denotes the determinant of $A$. This inequality can be applied to \eqref{second-hession-5-30}  to get  \begin{align}  p^n \big(\varphi_t^*(y)\big)^{n(p-1)}  \det\big(\nabla^2\varphi^*_t(y)\big) \leq   \det \big(\nabla^2(\varphi_t^*)^p(y)\big) =\det\big(\nabla^2(\varphi^*)^p(y)+t\nabla^2(\psi^*)^p(y)\big). \label{second-hession-5-30-1} \end{align} 
	As  $g=e^{-\psi}\in \A'_0$ is an admissible $p$-perturbation for $f$,  there exists a constant $c>0$ such that $\phi=(\varphi^*)^p-c (\psi^*)^p$ is convex. Hence $\nabla^2  (\psi^*)^p =\frac{1}{c} \nabla^2  (\varphi^*)^p -\frac{1}{c} \nabla^2 \phi.$  It follows from \eqref{compatible-1},  \eqref{matrix-ineq}  and \eqref{second-hession-5-30-1}   that, for $y\neq o$ and $p>1$,  
	\begin{align*} 
	\det\big(\nabla^2\varphi^*_t(y)\big)
	& \leq  p^{-n} \big(\varphi_t^*(y)\big)^{n(1-p)}    \det\Big(\nabla^2 (\varphi^*)^p (y)+t\nabla^2 (\psi^*)^p (y)\Big)\\
	& \leq  p^{-n} \big(\varphi^*(y)\big)^{n(1-p)}   \det\Big(\nabla^2 (\varphi^*)^p (y)+\frac{t}{c} \nabla^2  (\varphi^*)^p(y) -\frac{t}{c} \nabla^2 \phi(y)\Big)\\ 
	&\leq p^{-n} \big(\varphi^*(y)\big)^{n(1-p)}  \det\Big(\nabla^2 (\varphi^*)^p (y)+\frac{t}{c} \nabla^2  (\varphi^*)^p(y) \Big)\\ 
	&\leq \bigg(\frac{t+c}{cp}\bigg)^{n} \big(\varphi^*(y)\big)^{n(1-p)}  \det\Big(\nabla^2 (\varphi^*)^p (y) \Big)\\ 
	& \leq k\bigg(\frac{t+c}{cp}\bigg)^{n} \det\big(\nabla^2 \varphi^* (y) \big),
	\end{align*} where the constant $k>0$ is given by \eqref{compatible-1}. Let $k_0=k\big(\frac{1+c}{cp}\big)^{n}.$ Hence, if $t\in (0, 1)$, one gets \begin{align}\label{compatible-2} \det \big(\nabla^2\varphi^*_t(y)\big)\leq k_0  \det\big(\nabla^2\varphi^* (y) \big). \end{align}
	
	For any $o\neq y\in \R^n$,  by \eqref{def-dual-2}, \eqref{first-order-deriv} and Lemma \ref{vainshat0}, one gets \begin{align*}  \frac{\,d }{\,dt}\Big(\big( \varphi^*_t(y)\big)^p-\Big\langle y, \big(\varphi^*_t(y)\big)^{p-1}\nabla \varphi^*_t(y)\Big\rangle\Big)  &= \big( \psi^*(y)\big)^p-\Big\langle y, \big(\psi^*(y)\big)^{p-1}\nabla \psi^*(y)\Big\rangle\\ &= \big( \psi^*(y)\big)^{p-1} \Big(\psi^*(y)-\big\langle y,  \nabla \psi^*(y)\big\rangle \Big) \\&=-\big( \psi^*(y)\big)^{p-1}  \psi(\nabla \psi^*(y)  \leq 0.  \end{align*} 	
  Together with \eqref{varphi-t-star} and \eqref{p-add-compare-1}, one gets  that, for $p>1$, $o\neq y\in\R^n,$ and $t\in (0, 1)$,
 \begin{align}  
 \varphi^*_t(y)-\big\langle y, \nabla \varphi^*_t(y)\big\rangle 
 & \leq \big(\varphi^*_t(y)\big)^{1-p}\ \Big( \big(\varphi^*(y)\big)^p-\Big\langle y, \big(\varphi^*(y)\big)^{p-1}\nabla \varphi^*(y)\Big\rangle\Big) \nonumber \\  &=\bigg(\frac{\varphi^*(y)} {\varphi^*_t(y)}\bigg)^{p-1}  \big(\varphi^*(y)- \langle y, \nabla \varphi^*(y)\rangle\big) \nonumber \\
 &\leq \Big(\frac{c}{1+c}\Big)^{\frac{p-1}{p}}  \big(\varphi^*(y)- \langle y, \nabla \varphi^*(y)\rangle\big)\leq 0. 
\label{s-compare-5} 
\end{align}
 Formulas \eqref{varphi-t-star} and \eqref{p-add-compare-1}  also yield that, for any $t>0$,  \begin{eqnarray} (\psi^*)^p\leq c^{-1}  \varphi^* (\varphi_t^*)^{p-1}.\label{p-add-com-222}  \end{eqnarray} 

Now we are ready to check the interchange of orders of limit and integration in  \eqref{Lagrange-5-31-22}. Combining \eqref{def-dual-2},   \eqref{compatible-2}   
\eqref{s-compare-5}, 
\eqref{p-add-com-222}, one has,  for $s\in (0, 1)$ and $o\neq y\in \R^n$, \begin{align*} h(y)&=k_0 c^{-1} \varphi^*(y)  \exp\bigg(\Big(\frac{c}{1+c}\Big)^{\frac{p-1}{p}}  \big(\varphi^*(y)- \langle y, \nabla \varphi^*(y)\rangle\big)\bigg) \det \big(\nabla^2\varphi^*(y)\big) \\ &\geq \big(\psi^*(y)\big)^p \Big(\varphi^*_{s}(y)\Big)^{1-p}e^{-\varphi_s(\nabla\varphi_s^*(y))} \det \big(\nabla^2\varphi^*_s(y)\big).\end{align*} The function $h$ is  integrable by \eqref{const-mul-dual} and \eqref{finitezza-1}. Indeed, a calculation similar to \eqref{Lagrange-5-31-22} leads that \begin{align*} \int_{\R^n\setminus\{o\}} h(y)\,dy &=
 k_0 c^{-1}  \int_{\R^n\setminus\{o\}}  \varphi^*(y)  \exp\bigg(\Big(\frac{c}{1+c}\Big)^{\frac{p-1}{p}}  \big(\varphi^*(y)- \langle y, \nabla \varphi^*(y)\rangle\big)\bigg) \det \big(\nabla^2\varphi^*(y)\big)\,dy\\ &= 
 k_0 c^{-1}  \int_{\R^n}   \varphi^*(\nabla\varphi(x))  \exp\Big(-\Big(\frac{c}{1+c}\Big)^{\frac{p-1}{p}}  \varphi(x) \Big) \,dx\\ &=k_0 c^{-1}   \Big(\frac{c}{1+c}\Big)^{\frac{1-p}{p}} \int_{\R^n}   \widetilde{\varphi}^*(\nabla\widetilde{\varphi}(x))  e^{- \widetilde{\varphi}(x)}  \,dx \in (-\infty, \infty),
\end{align*} where $\widetilde{\varphi}=(\frac{c}{1+c})^{\frac{p-1}{p}}  \varphi$.  Therefore, the dominated convergence theorem can be applied to  \eqref{Lagrange-5-31-22} and the desired formula \eqref{tesi'} holds.  \end{proof}

The condition \eqref{compatible-1} is indeed natural and many widely used functions do satisfy this condition, for example, the Gaussian function $e^{-|x|^2/2}$. The following results give  some convenient ways to check condition  \eqref{compatible-1}. 

\begin{corollary} \label{verify-condition-1}  Assume that $\varphi\in \C^2_+(\R^n\setminus\{o\})$. Then \eqref{compatible-1} holds if any one of the following holds: 
	\vskip 2mm \noindent i) for some $\alpha\in [0, 1)$,  the function  $\frac{1}{\alpha} (\varphi^*)^{\alpha}$  (understood as $\log \varphi^*$ when $\alpha=0$) is convex;  \vskip 2mm  \noindent ii) there exists a constant $k_1$, such that, for any $y\in \R^n\setminus\{o\}$, \begin{align} \Big\langle \nabla \varphi^*(y), \big(\nabla^2\varphi^*(y)\big)^{-1} \nabla \varphi^*(y)\Big\rangle \leq k_1 \varphi^*(y).\label{condition-5-31-1} \end{align} 
\end{corollary} 

\begin{proof} For $p>1$, condition \eqref{compatible-1} is equivalent to,  for any $y\neq o$,  
	\begin{equation} \label{compatible-1-22}  H(y)= \det\Big(\nabla^2 \varphi^* (y)+
	(p-1)
	(\varphi^*(y))^{-1} \nabla\varphi^*(y) \otimes \nabla\varphi^* (y)\Big)\leq k  p^{-n} \det\big(\nabla^2 \varphi^* (y) \big).
	\end{equation}
	
	\vskip 2mm \noindent i) Let $\alpha\in [0, 1)$ and the function  $\frac{1}{\alpha} (\varphi^*)^{\alpha}$  be convex. For any $o\neq y\in \R^n$,  $$A(y)=\nabla^2 \varphi^* (y)+(\alpha-1)(\varphi^*(y))^{-1} \nabla\varphi^*(y)\otimes \nabla\varphi^*(y)  $$ is a positive semi-definite matrix.  Therefore, \eqref{matrix-ineq}  yields 
	\begin{align*}
	H(y) &=\det\!\Big(\! 
	\Big(1+\frac{p-1}{1-\alpha}\Big)
	\nabla^2 \varphi^* (y)+
	\frac{p-1}{\alpha-1} 
	A (y)\!\Big) \leq  \Big(1+   \frac{p-1}{1-\alpha}\Big)^n \det\Big(\nabla^2 \varphi^* (y)\Big) . 
	\end{align*} 
	Hence, \eqref{compatible-1-22}  and   condition \eqref{compatible-1} hold true. 
	
	\vskip 2mm \noindent ii) It can be calculated that $\mathbb{I}_n+z\otimes z$ for any $z\in \R^n$ has its determinant to be $1+|z|^2$. Hence, for any $o\neq y\in \R^n$,  \eqref{condition-5-31-1}  yields \begin{align*}  H(y) &=\det\big(\nabla^2 \varphi^* (y)\big)\Big(1+  (p-1) (\varphi^*(y))^{-1} \Big\langle \nabla \varphi^*(y), \big(\nabla^2\varphi^*(y)\big)^{-1} \nabla \varphi^*(y)\Big\rangle\Big) \\ &\leq \det\big(\nabla^2 \varphi^* (y)\big)\Big(1+k_1  (p-1) \Big). \end{align*}  This implies \eqref{compatible-1-22}  and hence condition \eqref{compatible-1} holds ture. 
\end{proof} 

\section{The $L_p$ Minkowski problem for log-concave functions}\label{functional p Minkowski problem} \setcounter{equation}{0}
This section aims to investigate the $L_p$ Minkowski problem for log-concave functions. Actually Theorem \ref{variationformula} suggests a new measure for log-concave functions. Let $\L^+=\{\varphi\in \L: \varphi\geq 0\}$, $\Omega_{\varphi^*}=\{y\in \R^n:  0<\varphi^*(y)<+\infty\}$ and $\widetilde{\Omega}_{\varphi^*}=\{y\in \R^n:  \varphi^*(y)=0\}.$ The set $\Omega_{\varphi^*} $ is always assumed to be nonempty. The subscript $\varphi^*$ is often omitted if there is no confusion. 

The  $L_p$  surface area measure of $f$ can be defined as follows. 
\begin{definition} \label{p-def-moment-measure}  Let $f=e^{-\varphi}$ be a log-concave function with $\varphi \in \L$ such that $\varphi^*\in \L^+$ and $\Omega$ is nonempty.  For $p\in \R$, the $L_p$  surface area measure of $f$, denoted by $\mu_p(f, \cdot)$, is the Borel measure on $\Omega$ such that  \begin{equation}\label{moment-form-p-6-1} \int_{\Omega}g(y)\, d\mu_p (f, y)=\int_{\{x\in \dom(\varphi):\   \nabla \varphi(x)\in \Omega \}}g(\nabla \varphi(x))(\varphi^*(\nabla \varphi(x)))^{1-p} e^{-\varphi(x)} \,dx\end{equation} holds for every Borel function $g$ such that $g \in L ^ 1 (\mu_p (f,\cdot))$  or $g$ is non-negative. \end{definition}	

 In general,  $\,d\mu_{p}(f, \cdot)=(\varphi^*)^{1-p} \,d\mu_1(f, \cdot)$  on  $\Omega$.  The $L_p$ surface area measure of $f$ for $p=1$ in Definition \ref{p-def-moment-measure}  is  the restriction of the surface area measure $f$ given in \eqref{moment-form-1} on $\Omega$, i.e.,  $\mu_1(f, \cdot)=\mu(f, \cdot)|_{\Omega}$. If  the Lebesgue measure of $\widetilde{\Omega}$ is zero,  $\mu_1(f, \cdot)$ can be extended to $\R^n$ (more precisely the interior of $\dom(\varphi^*)$), and reduces to $\mu(f, \cdot)$.   The measure $\mu_1(f, \cdot)$ is always finite for $f\in \A$. 
As $\varphi^*=0$ on $\widetilde{\Omega}$, one can even extend $\mu_p(f, \cdot)$ for $p<1$ to $\mathrm{int}(\dom(\varphi^*))$. 
When $p=0$, one has the $L_0$ (or  the logarithmic) surface area measure of $f$ which is again a finite measure for $f\in \A$ based on \eqref{finitezza-1}. 

A natural problem to characterize the $L_p$  surface area measure of  a log-concave function $f$ can be formulated as follows.  
\begin{problem}[The $L_p$ Minkowski problem for log-concave functions] \label{problem-lp-6-2}  Let $\nu$ be a  finite nonzero Borel measure on $\mathbb{R}^n$ and $ p\in \R$. Find the necessary and/or sufficient conditions on $\nu$, so that,    $$\nu=\tau \mu_{p}(f,\cdot)\ \ \ \mathrm{or}\ \ \ (\varphi^*)^{p-1}\nu =\tau\mu_1(f, \cdot)$$ hold for some log-concave function $f=e^{-\varphi}$ and $\tau\in\R$.\end{problem} 

When $p=1$ and the Lebesgue measure of $\widetilde{\Omega}$ is zero, Problem \ref{problem-lp-6-2} for $p=1$  reduces to the Minkowski problem for moment measures investigated in \cite{CK15} by Cordero-Erausquin and  Klartag. See \cite[Section 7]{CF13} for a full formulation to this problem. In this case, a solution has been provided in \cite{CK15}, including \cite[Proposition 1]{CK15} for necessity and  \cite[Theorem 2]{CK15} for sufficiency and uniqueness.

When $f=e^{-\varphi}$ is smooth enough so that $\nabla \varphi: \mathrm{int}(\dom(\varphi)) \rightarrow \mathrm{int}(\dom(\varphi^*))$ is smooth and bijective, then formula \eqref{def-dual-2} and Definition \ref{p-def-moment-measure} deduce  that, by letting $y=\nabla \varphi(x)$,  
\begin{align*}
\int_{\Omega}g(y)\, d\mu_p (f, y)&=\int_{\{x\in \R^n:\   \nabla \varphi(x)\in \Omega \}}g(\nabla \varphi(x))(\varphi^*(\nabla \varphi(x)))^{1-p} e^{-\varphi(x)} \,dx \\
&=\int_{\Omega}g(y) \varphi^*(y)^{1-p}\det(\nabla^2\varphi^*(y))e^{\varphi^*(y)-\langle y, \nabla\varphi^*{y}\rangle}\,dy, 
\end{align*} holds for every Borel function $g$ such that $g \in L ^ 1 (\mu_p (f,\cdot))$  or $g$ is non-negative. That is,   $\mu_{p}(f,\cdot)$ for $p\in\R$  is absolutely continuous with respect to the Lebesgue measure and satisfies that 
\begin{align*}
\frac{d\mu_{p}(f, y)}{dy}=\varphi^*(y)^{1-p} e^{\varphi^*(y)-\langle y, \nabla\varphi^*{y}\rangle} \det(\nabla^2\varphi^*(y)) \ \ \mathrm{for} \ \  y\in \Omega. \end{align*}  Consequently,  if $\nu$ admits a density function with respect to the Lebesgue measure, say  $\,d\nu(y)=h(y)\,dy,$  then finding a solution to the $L_p$ Minkowski problem for log-concave functions requires to obtain a (smooth enough) convex function $\varphi$ satisfying the following Monge-Amp\`{e}re equation:  
\begin{align*}
h(y) &=  \tau  \varphi^*(y)^{1-p} e^{\varphi^*(y)-\langle y, \nabla\varphi^*{y}\rangle} \det(\nabla^2\varphi^*(y)), \ \ \ \mathrm{for} \ y\in \Omega,   
\end{align*}  where $\tau\in\R$ is a constant.

Our main goal in this section is to provide a solution to Problem \ref{problem-lp-6-2} for $p>1$.  
Let $M_{\nu}$ be the interior of the convex hull of the support of the Borel measure $\nu$.   For convenience, denote by $\mathscr{M}$  the set of all {\em even finite  nonzero  Borel measures} on $\R^n$, such that, if $\nu\in \mathscr{M}$, then  $\nu$ is not supported in a lower-dimensional subspace of $\R^n$, $\nu(M_{\nu}\setminus L)>0$ holds for any bounded convex set $L\subset\R^n$, and 
\begin{equation} \label{p-th-moment} \int_{\R^n} |x|^p\, d\nu(x)<\infty.\end{equation} Note that the conditions for $\nu\in \mathscr{M}$ are all natural. Indeed, the assumption that $\nu$ is not supported in a lower-dimensional subspace of $\R^n$ is essential in the solution to the Minkowski problem for moment measures in \cite{CK15}. The requirement for \eqref{p-th-moment}  is to guarantee that the optimization problem \eqref{mini problem} is not taken over an empty set. Finally, the condition that $\nu(M_{\nu}\setminus L)>0$ holds for any bounded convex set $L\subset\R^n$ is to guarantee  that $\nu(M_{\nu}\setminus \widetilde{\Omega}_{\varphi})>0$ for any $\varphi$ and thus avoid that $\Omega_{\varphi}$ being either empty or a null set.  

\begin{theorem}  \label{solution-lp-mink-6-2} Let $\nu\in \mathscr{M}$. For  $p>1$,  there exists an even log-concave function  $f=e^{-\varphi}$, such that,  $\varphi\in \C$  is even and lower semi-continuous, $\varphi^*\in \L^+$, and
	\begin{eqnarray*}
		\nu =\tau (\varphi^*)^{1-p} \mu_{1}(f,\cdot) =\tau \mu_{p}(f, \cdot)\ \ \ \mathrm{on} \ \ \Omega,
	\end{eqnarray*}  where the constant $\tau$ takes the following formula: 
	$$ \tau= \frac{\int_{\Omega} (\varphi^*(y))^{p-1}\,d\nu(y)}  {\int_{\Omega} \,d \mu_{1}(f,y) } = \frac{\int_{\Omega} \,d\nu(y)}  {\int_{\Omega} \,d \mu_{p}(f,y) } .$$ \end{theorem}

Before we prove Theorem \ref{solution-lp-mink-6-2}, we need some preparation.  Let $p>1$ and $\nu\in \mathscr{M}$. Denote by  $L_{p, e} (\nu)$ the set of even non-negative  functions on $\R^n$ which have finite $L^p$ norm with respect to the measure $\nu$ and whose function values are $0$ at $o$.  For $\phi\in L_{p, e} (\nu)$, let  
\begin{align*}
\mathbf{\Phi}_{p,\nu}(\phi)=\frac{1}{p}\int_{\R^n} \!\!  \phi(x)^pd\nu(x)-\log J(e^{-\phi^*}).
\end{align*} To find a solution to Problem \ref{problem-lp-6-2}, 
one needs to search for a solution to   
\begin{eqnarray}\label{mini problem}
\Theta=\inf\Big\{\mathbf{\Phi}_{p,\nu}(\phi):   \phi \in L_{p, e}(\nu) \ \ \mathrm{and} \ \ 0< J(e^{-\phi^*})<\infty  \Big\}.
\end{eqnarray}

The following lemma shall be needed to solve \eqref{mini problem}. Similar results for $p=1$ can be found in \cite[Lemmas 14 and 15]{CK15}.

\begin{lemma}\label{lemma1.3} Let $p>1$ and $\nu\in \mathscr{M}$.  Then there exists a constant   $c_{\nu}>0$, such that, for any  $\phi \in L_{p, e}(\nu)\cap \C$ satisfying $0<J(e^{-\phi^*})<\infty$, the following holds: 
	\begin{eqnarray} \int_{\R^n}\phi(x)^p\,d\nu(x)\geq   c_{\nu}\big(J(e^{-\phi^*})\big)^{\frac{1}{n}}-\int_{\R^n} \,d\nu(x).\label{equation-6-10}\end{eqnarray}
\end{lemma}
\begin{proof} Let $p>1$ and $\nu\in \mathscr{M}$. Assume that $0<J(e^{-\phi^*})<\infty$. For any  $\phi \in L_{p, e}(\nu)\cap \C$, $\phi^*$ must be an even convex function.  It is well-known that there exist constants $c_1, c_2>0$ such that for any $n\in \mathbb{N}$ and any even log-concave function $f=e^{-\phi}$ with $J(e^{-\phi})\in (0, \infty)$, the following holds: \begin{align}\label{santalo-func-6-10} c_1^n\leq J(e^{-\phi}) J(e^{-\phi^*})\leq c_2^n. \end{align} We refer the readers to \cite[Theorem 1.1]{KM05} for more details on inequality \eqref{santalo-func-6-10}. The sharp constant for $c_2$  is $2\pi$ and the upper bound is indeed the Blaschke-Santal\'{o} inequalities for (even) log-concave functions, see e.g.,  \cite{AKM04, Ball88-2, FM07,FM08}.  Applying inequality \eqref{santalo-func-6-10} to the log-concave function $e^{-\phi^*}$, one gets that $0<J(e^{-\phi})<\infty$.  	
	
	For  $\phi \in L_{p, e}(\nu)\cap \C$ satisfying $0<J(e^{-\phi^*})<\infty$, let
	$ K_{\phi}=\{x\in \mathbb{R}^n: \phi(x)\leq 1\}.$ Clearly,  $K_{\phi}$ is an origin-symmetric bounded  convex set  containing the origin $o$ in its interior (due to $\phi(o)=0$).  As  $0< J(e^{-\phi})<\infty$, one has $V(K_{\phi})<\infty$. Denote by $\mathrm{vrad}(\phi)$ the volume radius of $K_{\phi}$, that is, $$\mathrm{vrad}(\phi)=\bigg(\frac{V(K_{\phi})}{V(\ball)}\bigg)^{1/n}\in (0, \infty).$$ Note that $K_{\phi}$ cannot contain an Euclidean ball whose radius is  greater than $\mathrm{vrad}(\phi)$. Consequently,  one may find a vector $\theta_0\in S^{n-1}$ such that \begin{eqnarray}\label{an upper bound}
	\sup_{x\in K_{\phi}}|\langle x,\theta_0\rangle|=\sup_{x\in K_{\phi}}\langle x,\theta_0\rangle \leq \mathrm{vrad}(\phi).
	\end{eqnarray}
	The fact that $\phi$ is convex yields that, for $p>1$ and for any $x\in \R^n$ with $|\langle x,\theta_0\rangle|\geq   \mathrm{vrad}(\phi)$, \begin{eqnarray}\label{an upper bound2}
	\phi^p\left(\frac{ \mathrm{vrad}(\phi)}{|\langle x,\theta_0\rangle|}x\right)&\leq& \frac{ \mathrm{vrad}(\phi)}{|\langle x,\theta_0\rangle|}\phi^p(x).\end{eqnarray}  According to \eqref{an upper bound} and \eqref{an upper bound2}, if $x\in \R^n$ such that $|\langle x,\theta_0\rangle|\geq  \mathrm{vrad}(\phi)$, then $\phi(x)\geq 1$ and 
	\begin{align}\label{an upper bound3}
	\phi^p(x)\geq \frac{|\langle x,\theta_0\rangle|}{\mathrm{vrad}(\phi)} \phi^p\left(\frac{\mathrm{vrad}(\phi)}{|\langle x,\theta_0\rangle|}x\right)
	\geq \frac{|\langle x,\theta_0\rangle|}{\mathrm{vrad}(\phi)} 	\geq \frac{|\langle x,\theta_0\rangle|}{\mathrm{vrad}(\phi)} -1.
	\end{align} Indeed, 	(\ref{an upper bound3}) holds for all $x\in \R^n$ as it is trivial to have (\ref{an upper bound3}) for those $x\in \R^n$ such that $|\langle x,\theta_0\rangle|<  \mathrm{vrad}(\phi)$, due to $\phi(x)\geq 0$.  Consequently, 
	\begin{eqnarray*}
		\int_{\R^n}\phi^p(x)d\nu(x)\geq \frac{1}{ \mathrm{vrad}(\phi) }\int_{\R^n}|\langle x,\theta_0\rangle|d\nu(x)-\int_{\R^n} \,d\nu(x)\geq \frac{m_{\nu}}{ \mathrm{vrad}(\phi) }-\int_{\R^n} \,d\nu(x),
	\end{eqnarray*}
	where $m_{\nu}=\inf_{\theta\in S^{n-1}}\int_{\R^n}|\langle x,\theta\rangle|\, d\nu(x)$. Note that $m_{\nu}\in (0, \infty)$ is a direct consequence of the conditions on $\nu\in \mathscr{M}$, in particular, the function $\theta \mapsto \int_{\R^n}|\langle x,\theta\rangle|\, d\nu(x)$ is positive and continuous on $\theta$ due to the dominated convergence theorem.  On the other hand,  
	\begin{eqnarray*} \int_{\R^n}e^{-\phi(x)}\,dx \geq\int_{K_{\phi}}e^{-\phi(x)}\,dx \geq\frac{V(K_{\phi})}{e}=\frac{V(\ball)}{e}  \mathrm{vrad}(\phi)^n. \end{eqnarray*}  Therefore, an application of inequality \eqref{santalo-func-6-10} with  $c_2=2\pi$  immediately yields 
	\begin{align*}  \int_{\R^n}\phi^p(x)d\nu(x) &\geq  \frac{ (V(\ball))^{1/n} m_{\nu}}{e^{1/n}} \bigg(\int_{\R^n} e^{-\phi(x)}\,dx\bigg)^{-\frac{1}{n}}-\int_{\R^n} \,d\nu(x)  \nonumber \\ &
	\geq   c_{\nu} \bigg(\int_{\R^n} e^{-\phi^*(x)}\,dx\bigg)^{\frac{1}{n}}-\int_{\R^n} \,d\nu(x), \end{align*} 
	by letting  $c_{\nu}=\frac{ (V(\ball))^{1/n} m_{\nu}}{2\pi e^{1/n}}$. This concludes the desired formula \eqref{equation-6-10}. 
\end{proof}

It has been proved in  \cite[Lemma 16]{CK15} that,  if $\nu$ is a nonzero finite Borel measure on  $\R^n$ that is not supported in a lower-dimensional subspace of $\R^n$,  for $x_0\in M_{\nu}$,  then there exists a constant $C_{\nu, x_0}>0$ with the following property: $$
\phi(x_0)\leq C_{\nu,x_0}\int_{\R^n}\phi(x) \,d\nu(x)
$$ holds for any $\nu$-integrable convex function $\phi: \R^n\rightarrow [0,\infty]$. This can be applied to $\nu\in \mathscr{M}$ and $p>1$ to get that, for $x_0\in M_{\nu}$,  then there exists a constant $C_{\nu, x_0}>0$, such that, \begin{align}\label{estimate-p-6-12}
\phi^p(x_0)\leq C_{\nu,x_0}\int_{\R^n}\phi^p(x) \,d\nu(x)
\end{align}  holds for any  $\phi\in  \C\cap L_{p, e}(\nu)$.  

We shall need the following lemma given by \cite[Theorem 10.9]{Roc70}.  

\begin{lemma}\label{Rockafellar's result} Let $C$ be a relatively open convex set, and let $\phi_1, \phi_2,\cdots,$ be a sequence of finite convex functions on $C$. Suppose that the real number  $\phi_1(x), \phi_2(x),\cdots,$
	is bounded for each $x\in C$. It is then possible to select a subsequence of  $\phi_1, \phi_2,\cdots,$ which converges uniformly  on  closed  bounded subsets of $C$ to some finite convex function $\phi$.
\end{lemma}

We are now ready to prove  the following lemma. The case $p=1$ has been discussed in \cite[Lemma 17]{CK15}. 

\begin{lemma}\label{existence}Let $p>1$ and $\nu\in \mathscr{M}.$  Assume that $\phi_l\in \C\cap L_{p, e}(\nu)$  for any $l\in \mathbb{N}$  satisfy 
	\begin{eqnarray}\label{condition1}
	\sup_{l\in \mathbb{N}} \int_{\R^n} \phi_l(x)^p\, d\nu(x)<+\infty.
	\end{eqnarray}
	Then there exists a subsequence $\{\phi_{l_{j}}\}_{j\in \mathbb{N}}$ of $\{\phi_l\}_{l\in \mathbb{N}}$ and a non-negative  convex function $\phi \in \C\cap L_{p, e}(\nu)$, such that,
	\begin{eqnarray}\label{liminf}
	\int_{\R^n} \phi(x)^p \,d\nu(x) \leq \liminf_{j\rightarrow\infty} \int_{\R^n} \phi_{l_j}(x)^p \,d\nu(x) \ \ \mathrm{and}\ \ 
	\int_{\R^n}e^{-\phi^*(x)}\,dx\geq \limsup_{j\rightarrow\infty}\int_{\R^n}e^{-\phi_{l_j}^*(x)}\,dx.
	\end{eqnarray}
\end{lemma}

\begin{proof}  We prove this lemma following the ideas of the proof of \cite[Lemma 17]{CK15}, with emphasis on the difference and modification. 
	
	As $\nu\in \mathscr{M}$ is an even measure which is not supported in a lower-dimensional subspace of $\R^n$, the open set $M_{\nu}$ is nonempty and origin-symmetric with $o\in M_{\nu}$.  By \eqref{estimate-p-6-12} and \eqref{condition1},  for any $x\in M_{\nu}$, one has $
	\sup_{l\in \mathbb{N}} \phi_l(x)^p<+\infty. $   We would like to mention that, if $\phi\in L_{p, e} (\nu)$, then $\phi$ is finite near the origin and $\dom(\phi)\supseteq M_{\nu}$.  Lemma \ref{Rockafellar's result} can be applied to $\phi_l^p$ and $C=M_{\nu}$ to obtain the existence of a  subsequence $\{\phi_{l_j}\}_{j\in \mathbb{N}}$ of $\{\phi_l\}_{l\in \mathbb{N}}$, which converges to an even convex function $\phi: M_{\nu} \rightarrow \mathbb{R}$ pointwisely on $M_{\nu}$ and also uniformly on any closed bounded subset of $M_{\nu}$. The finiteness of $\phi$ on $M_{\nu}$ implies the continuity of $\phi$ on $M_{\nu}$. Moreover, $\phi$ is non-negative in $M_{\nu}$ and achieves its minimum at the origin with $\phi(o)=0$ (as $\phi_l(o)=0$ for each $l\in \mathbb{N}$). The function $\phi: M_{\nu}\rightarrow \R$ can be extended on $\R^n$, and the new function will  still be denoted by $\phi$. That is, let $\phi(x)=+\infty$  for $x\notin \overline{M_{\nu}}$;  while $\phi(x)=\lim_{\lambda\rightarrow 1^-} \phi(\lambda x)$ if  $x\in \partial M_{\nu}$. The limit in the latter case always exists, although it may be $+\infty$, due to the fact that the function $\lambda\mapsto \phi(\lambda x)$ is increasing on $\lambda\in (0, 1)$ following from the convexity of $\phi$ and $\phi(o)=0$.    Moreover, for any $x\in \overline{M_{\nu}}$, $\phi(\lambda x)$ is increasing to $\phi(x)$ as $\lambda$ is increasing to $1$. This shows that $\phi: \R^n\to \R$ is an even, non-negative, and lower semi-continuous convex function with $\phi(o)=0$.
	
	Note that the support of $\nu$ is a subset of $\overline{M_{\nu}}$.   It follows from Fatou's lemma and $\phi_{l_j}\rightarrow \phi$ pointwisely in $M_{\nu}$ that, for any given $\lambda\in (0, 1)$, one has \begin{align*}
	\int_{\R^n} \phi(\lambda x)^p\,d\nu(x) = \int_{M_{\nu}} \phi(\lambda x)^p\,d\nu(x)& \leq  \liminf_{j\rightarrow\infty} \int_{M_{\nu}}\phi_{l_{j}}(\lambda x)^pd\nu(x) \nonumber \\ & \leq   \liminf_{j\rightarrow\infty} \int_{\mathbb{R}^n}\phi_{l_{j}}(x)^pd\nu(x)<+\infty, 
	\end{align*} where the second inequality again follows from the monotonicity of  the function $\lambda\mapsto \phi_{l_j}(\lambda x).$ Similarly, by the monotone convergence theorem, one can also obtain that  \begin{eqnarray*}
		\int_{\R^n}\phi(x)^p\,d\nu(x)= \lim_{\lambda\rightarrow 1^{-}}\int_{\R^n}\phi(\lambda x)^p\,d\nu(x)\leq  \liminf_{j\rightarrow\infty} \int_{\mathbb{R}^n}\phi_{l_{j}}(x)^p\,d\nu(x)<+\infty.
	\end{eqnarray*}
	This completes the proof of the first argument in (\ref{liminf}).   In particular, $\phi\in L_{p, e}(\nu)\cap \C$,  $o\in  M_{\nu}\subseteq \dom(\phi)$ and hence   $J(e^{-\phi^*})<\infty$.

	The second argument in \eqref{liminf} indeed follows immediately from the proof of \cite[Lemma 17]{CK15}. For completeness, a brief explanation extracted from \cite[p.~3861-3862]{CK15} is provided here. First of all,  for any $y \in \mathbb{R}^n$, one has  $\phi^*(y)=\sup_{i\in \mathbb{N}}\{\langle x_i,y\rangle-\phi(x_i)\} $ where  $\{x_i\}_{i\in \mathbb{N}}$ is a dense sequence in $M_{\nu}$.  When $j$ is large enough, the origin lies in the interior of  the convex hull of $\{x_1,\cdots, x_j\}$, and this in turn implies that  $\exp(-h_j)$ with $h_j(y)=\sup_{1\leq i\leq j}\{\langle x_i,y\rangle-\phi(x_i)\})$  is integrable. Clearly, $h_j$  is increasing to $\phi^*$ as $j$ is increasing to $\infty$. By the monotone convergence theorem, one gets 
	\begin{eqnarray*}\int_{\mathbb{R}^n}e^{-\phi^*(y)}\,dy=\lim_{j\rightarrow\infty}\int_{\mathbb{R}^n}e^{-h_j(y)}\,dy.\end{eqnarray*} Let $\varepsilon > 0$. An integer $j_0$ (depending only on $\varepsilon$) can be found to have 		\begin{eqnarray}\label{condition4}
	0\leq \int_{\mathbb{R}^n}e^{-h_{j_0}(y)}\,dy-\int_{\mathbb{R}^n}e^{-\phi^*(y)}\,dy <\varepsilon.
	\end{eqnarray} The pointwise convergence of $\phi_{l_{j}}\rightarrow \phi$ (on $\{x_1,\cdots, x_{j_0}\}$)  yields that, for sufficiently large $j$, $\phi^*_{l_j}(x)\geq h_{j_0}(x)-\varepsilon$ holds for all $x\in \mathbb{R}^n$. Together with (\ref{condition4}), the following holds: \begin{eqnarray*} 
		\int_{\mathbb{R}^n}e^{-\phi^*(y)}\,dy >\int_{\R^n} e^{-h_{j_0}(y)}\,dy -\varepsilon\geq e^{-\varepsilon}\limsup_{j\rightarrow\infty}\int_{\mathbb{R}^n}e^{-\phi_{l_{j}}^*(y)}\,dy-\varepsilon.
	\end{eqnarray*} The second argument in \eqref{liminf}  then follows by letting $\varepsilon\rightarrow 0$. 
\end{proof}

Now we are ready to prove our main result, i.e., Theorem \ref{solution-lp-mink-6-2}. The $L_p$ surface area measure for log-concave functions in general does not have homogeneity, hence solving the related Minkowski problem (i.e., Problem \ref{problem-lp-6-2}) usually requires more delicate analysis. Most of the time, such problems shall require to solve constrained optimization problems, and  the method of Lagrange multipliers as in \cite{GHWXY19, XY17, ZSY17} should be used; this method should work here for a proof of Theorem  \ref{solution-lp-mink-6-2}.  However, we find that the ideas in the proof of \cite[Theorem 2]{CK15} work well in our case for $p>1$. Therefore, we decide to adopt  the ideas in \cite[Theorem 2]{CK15} in the proof of Theorem \ref{solution-lp-mink-6-2}. Unfortunately, due to the lack of homogeneity for the $L_p$ surface area measure for log-concave functions for $p>1$, it is unlikely to have the uniqueness of solutions to Problem \ref{problem-lp-6-2}.

\begin{proof}[Proof of Theorem \ref{solution-lp-mink-6-2}] We will search a solution for the following optimization problem \eqref{mini problem}: 
	\begin{eqnarray*} 
		\Theta=\inf\Big\{\mathbf{\Phi}_{p,\nu}(\phi):   \phi \in L_{p, e}(\nu) \ \ \mathrm{and} \ \ 0< J(e^{-\phi^*})<\infty  \Big\}. \end{eqnarray*}
	According to $(\phi^*)^*\leq \phi$ and $((\phi^*)^*)^*=\phi^*$ for any $\phi \in L_{p, e}(\nu)$, it can be checked that  
	\begin{align*}
	\mathbf{\Phi}_{p,\nu}(\phi)=\frac{1}{p}\int_{\R^n} \!\!  \phi(x)^pd\nu(x)-\log J(e^{-\phi^*})\geq \frac{1}{p}\int_{\R^n} \!\!  ((\phi^*)^*)^p(x)d\nu(x)-\log J(e^{-\phi^*})=\mathbf{\Phi}_{p,\nu}((\phi^*)^*).
	\end{align*} 
	Consequently, to find a solution to the optimization problem \eqref{mini problem}, it is enough to focus on the class of functions $\phi \in L_{p, e}(\nu)\cap \C$ which are also lower semi-continuous.   
	
	We now claim that  the optimization problem \eqref{mini problem} is well-defined. First of all, \eqref{p-th-moment}  implies that $|x| \in L_{p, e}(\nu)\cap \C$. Note that $|x|$ is  also lower semi-continuous. It is well-known that $|x|^*=\mathbf{I}_{\ball}$. Therefore, $J(e^{-(|x|)^*})=V(\ball)$ is finite. This in turn yields that the infimum of \eqref{mini problem} is not taken over an empty set and hence $\Theta<\infty$.  On the other hand, $\Theta$ is bounded from below. To see this, by \eqref{equation-6-10}, for any lower semi-continuous function $\phi  \in L_{p, e}(\nu)\cap \C$ with $J(e^{-\phi^*})\in (0, \infty)$, one has  
	\begin{align} \mathbf{\Phi}_{p,\nu}(\phi)  = \int_{\R^n}\phi(x)^p\,d\nu(x)-\log J(e^{-\phi^*})  \geq H\big(J(e^{-\phi^*})\big)
	-\int_{\R^n} \,d\nu(x), \label{est-6-15-1} \end{align}
	where  $H(t)=  c_{\nu} t^{\frac{1}{n}} -\log t$ for $t\in (0, \infty)$.   It can be easily checked that $H(\cdot)$ achieves its minimum at  $t_0=(n c_{\nu}^{-1})^n$ and $H(t)\geq H(t_0)>-\infty$. Thus, $\Theta>-\infty$ and the optimization problem \eqref{mini problem} is well-defined. We also would like to mention that \begin{align}\label{limit-6-15-1}\lim_{t\rightarrow 0^+} H(t)=\lim_{t\rightarrow +\infty} H(t)=+\infty. \end{align} 
	
	Let $\{\phi_l\}_{l\in \mathbb{N}}\subset L_{p,e}(\nu) \cap \C$ be a minimizing sequence of lower semi-continuous functions such that $J(e^{-\phi_l^*})\in (0, \infty)$ for each $l\in \mathbb{N}$ and \begin{align*} \Theta=\lim_{l\to \infty}\mathbf{\Phi}_{p,\nu}(\phi_l)\leq \mathbf{\Phi}_{p,\nu}(|x|)<+\infty. \end{align*} 
	Without loss of generality,  we can always assume that \begin{eqnarray}\label{formula-6-15}
	\sup_{l\in \mathbb{N}}\mathbf{\Phi}_{p,\nu}(\phi_l)\leq \mathbf{\Phi}_{p,\nu}(|x|)+1<+\infty.\end{eqnarray} 
	Together with \eqref{est-6-15-1} and \eqref{limit-6-15-1},  one sees that 
	\begin{align}
	0<\inf _{l\in \mathbb{N}}J(e^{-\phi_l^*}) \leq \sup_{l\in \mathbb{N}}J(e^{-\phi_l^*}) <+\infty.
	\end{align} 
	Combining with \eqref{formula-6-15}, the following holds: \begin{eqnarray*}
		\sup_{l\in \mathbb{N}} \int_{\R^n} \phi_l(x)^p\, d\nu(x)<+\infty.
	\end{eqnarray*} This is exactly the condition \eqref{condition1}. Therefore, Lemma \ref{existence} can be applied to get a subsequence $\{\phi_{l_{j}}\}_{j\in \mathbb{N}}$ of $\{\phi_l\}_{l\in \mathbb{N}}$ and a non-negative  convex function $\phi_0 \in \C\cap L_{p, e}(\nu)$, such that \eqref{liminf} holds, namely, 
	\begin{eqnarray*}
		\int_{\R^n} \phi_0(x)^p \,d\nu(x) \leq \liminf_{j\rightarrow\infty} \int_{\R^n} \phi_{l_j}(x)^p \,d\nu(x) \ \ \mathrm{and}\ \ 
		\int_{\R^n}e^{-\phi_0^*(x)}\,dx\geq \limsup_{j\rightarrow\infty}\int_{\R^n}e^{-\phi_{l_j}^*(x)}\,dx.
	\end{eqnarray*} 
	According to  \eqref{liminf}, one immediately has \begin{align}  \mathbf{\Phi}_{p,\nu}(\phi_0)\leq \liminf_{l\to \infty}\mathbf{\Phi}_{p,\nu}(\phi_l)= \lim_{l\to \infty}\mathbf{\Phi}_{p,\nu}(\phi_l)= \Theta \leq \mathbf{\Phi}_{p,\nu}(\phi_0).\label{max-6-15-1}\end{align} 
	Hence, $\phi_0$ solves the optimization problem  \eqref{mini problem}. Moreover,  $0<J(e^{-\phi_0^*})<\infty$ following from \eqref{limit-6-15-1}. Inequality \eqref{santalo-func-6-10}  yields that $0<J(e^{-\phi_0})<\infty$ as well. In particular, \begin{align} \lim_{|x|\rightarrow +\infty} \phi_0(x)=+\infty.\label{inf-6-15-1}\end{align}

	Let $\widetilde{\Omega}_{\phi_0}=\{y\in \R^n: \phi_0(y)=0\}$ and $\Omega_{\phi_0}=\{y\in \R^n: 0<\phi_0(y)<\infty\}$.  By the facts that $\dom(\phi_0)\supseteq M_{\nu}$ and $\phi_0$ is even, one gets that $\phi_0$ is continuous on $M_{\nu}$ and both $\widetilde{\Omega}_{\phi_0}$ and $\Omega_{\phi_0}$ are origin-symmetric. Moreover, \eqref{inf-6-15-1} yields that  $\widetilde{\Omega}_{\phi_0}$ is a bounded closed  (due to the lower semi-continuity of $\phi_0$)  convex set. Recall that if $\nu\in \mathscr{M}$, then $\nu(M_{\nu}\setminus L)>0$ holds for any bounded convex set $L\subset\R^n$. Consequently, $\nu(M_{\nu}\setminus \widetilde{\Omega}_{\phi_0})=\nu(M_{\nu}\cap  \Omega_{\phi_0})>0$ and $$\int_{\R^n} \phi_0^p(y)\,d\nu(y)=\int_{\Omega_{\phi_0}} \phi_0^p(y)\,d\nu(y)\geq \int_{M_{\nu}\setminus \widetilde{\Omega}_{\phi_0}} \phi_0^p(y)\,d\nu(y) >0.$$ In particular, $M_{\nu}\setminus \widetilde{\Omega} _{\phi_0} \subset  \Omega _{\phi_0}$ as $M_{\nu} \subset \dom(\phi_0).$ Thus, $ \Omega _{\phi_0}$ is an open set whose Lebesgue measure is strictly positive.

	Let $g: \R^n \to \R$ to be any even compactly support continuous function such that the support of $g$, denoted by $\mathrm{supp}(g),$  is a proper subset of $\Omega_{\phi_0}$. Moreover, the compact set $\overline{\mathrm{supp}(g)}$ is contained in $\Omega_{\phi_0}.$ Let $\phi_t=\phi_0 +tg$ and $\phi_t$ is a continuous function on $\Omega_{\phi_0}$. Note that    $\phi_0>0$ on $\overline{\mathrm{supp}(g)}$ and hence $\min_{x\in \overline{\mathrm{supp}(g)}}\phi_0(x)>0$. This further yields  the existence of $t_0>0$ such that $\phi_t$ is non-negative on $\Omega_{\phi_0}$ for all $t\in [-t_0, t_0]$. It is easily checked that for all $t\in [-t_0, t_0]$, $x\in \R^n$ and $p>1$, one has $\phi_t^p(x)\leq 2^p (\phi_0^p(x)+|t g(x)|^p)$ and hence $$\int_{\R^n} \phi_t^p(x)\,d\nu(x)\leq   2^p \int_{\R^n}\phi_0^p(x)\,d\nu(x) + 2^p |t|^p \int_{\R^n} | g(x)|^p\,d\nu(x)<\infty.  $$ 
	This means that $\phi_t\in L_{p, e}(\nu)$ for all $t\in [-t_0, t_0]$.  Also note that $$\phi_0(x)-t_0 \max_{x\in \R^n} |g(x)| \leq \phi_t(x)\leq \phi_0(x)+t_0 \max_{x\in \R^n} |g(x)|.$$ 
	By \eqref{def-dual-1}, one obtains that, for all $t\in [-t_0, t_0]$,  \begin{equation*}
	\phi_0^*(y)-t_0 \max_{x\in \R^n} |g(x)|\leq \phi_t^*(y)\leq \phi_0^*(y)+t_0 \max_{x\in \R^n} |g(x)|.
	\end{equation*} This concludes that $J(e^{-\phi_t^*}) \in (0, \infty)$ for all $t\in [-t_0, t_0]$ as	
	$$e^{-t_0 \max_{x\in \R^n} |g(x)|} J(e^{-\phi_0^*})  \leq J(e^{-\phi_t^*})\leq e^{t_0 \max_{x\in \R^n} |g(x)|} J(e^{-\phi_0^*}).$$
	Consequently, for any even and compactly support continuous function $g: \R^n \to \R$ with $\overline{\mathrm{supp}(g)}\subsetneq \Omega_{\phi_0}$, there exists $t_0>0$ such that $ \mathbf{\Phi}_{p,\nu}(\phi_0)\leq  \mathbf{\Phi}_{p,\nu}(\phi_t) $ holds for all $t\in [-t_0, t_0]$. Thus, $\phi_0$ satisfies that \begin{align}\label{variation-6-15} \frac{\,d}{\,dt} \mathbf{\Phi}_{p,\nu}(\phi_t)\bigg|_{t=0}=\frac{\,d}{\,dt} \bigg(\frac{1}{p}\int_{\R^n} \!\!  \phi_t(x)^pd\nu(x)\bigg)\bigg|_{t=0}-\frac{\,d}{\,dt} \big(\log J(e^{-\phi_t^*})\big)\bigg|_{t=0} =0. \end{align} By the dominated convergence theorem, one can easily get that \begin{align} \frac{\,d}{\,dt} \bigg(\frac{1}{p}\int_{\R^n} \!\!  \phi_t(x)^pd\nu(x)\bigg)\bigg|_{t=0}  \nonumber &=\frac{\,d}{\,dt} \bigg(\frac{1}{p}\int_{\R^n} \!\!  \big(\phi_0(x)+tg(x)\big)^pd\nu(x)\bigg)\bigg|_{t=0}\\ &= \int_{\R^n} \!\!  g(x) \phi_0(x)^{p-1}d\nu(x). \label{variation-6-15-1} \end{align} On the other hand, for all $x\in \R^n$ when $\phi_0^*$ is differentiable, it holds that \begin{align}\label{variation-6-15-22}
	\frac{\,d}{\,dt}\phi_{t}^*(x)\bigg |_{t=0}=-g(\nabla \phi_0^*(x)).\end{align}  We refer the readers to Berman and Berndtsson \cite[Lemma 2.7]{BB13} for a short proof.  In other words, formula \eqref{variation-6-15-22} holds for almost all $x$ in the interior of $\dom(\phi_0^*)$.  The dominated convergence theorem then gives \begin{align} \frac{\,d}{\,dt} \big(\log J(e^{-\phi_t^*})\big)\bigg|_{t=0} = \frac{1}{J(e^{-\phi_0^*})} \int_{\R^n} g(\nabla \phi_0^*(x)) e^{-\phi_0^*(x)}\,dx. \label{variation-6-15-321}
	\end{align} Together with \eqref{moment-form-1} for $\varphi=\phi_0^*$,  \eqref{variation-6-15}, \eqref{variation-6-15-1}, and \eqref{variation-6-15-321}, one has 
	\begin{align} \int_{\Omega_{\phi_0}} \!\! g(y) \phi_0(y)^{p-1}d\nu(y)&= \int_{\R^n} \!\! g(y) \phi_0(y)^{p-1}d\nu(y) \nonumber \\ &=\frac{1}{J(e^{-\phi_0^*})} \int_{\R^n} g(\nabla \phi_0^*(x)) e^{-\phi_0^*(x)}\,dx \nonumber  \\&=\frac{1}{J(e^{-\phi_0^*})} \int_{\R^n}g(y)\,d\mu(e^{-\phi_0^*}, y) \nonumber  \\&=\frac{1}{J(e^{-\phi_0^*})} \int_{\Omega_{\phi_0}}  g(y)\,d\mu(e^{-\phi_0^*}, y) \nonumber \end{align}  for any  even and compactly support continuous function $g: \R^n \to \R$ with $\overline{\mathrm{supp}(g)} \subsetneq \Omega_{\phi_0}$. This concludes that, on $\Omega_{\phi_0}$, \begin{align}   \phi_0^{p-1} \nu =\frac{1}{J(e^{-\phi_0^*})}  \mu(e^{-\phi_0^*}, \cdot).  \label{final-form-6-15-1} \end{align}  Let $\varphi=\phi_0^*$. Then 
	$\varphi\in \C$  is an even lower semi-continuous convex function, $\varphi^*\in \L^+$, and
	\begin{eqnarray}
	\nu =\frac{1}{J(e^{-\varphi})}  (\varphi^*)^{1-p} \mu_{1}(e^{-\varphi},\cdot) = \frac{1}{J(e^{-\varphi})}  \mu_{p}(e^{-\varphi}, \cdot)\ \ \ \mathrm{on} \ \ \Omega_{\varphi^*}. \label{equation-final-6-15} 
	\end{eqnarray}  By taking the integration from both sides of \eqref{final-form-6-15-1} or \eqref{equation-final-6-15} on $\Omega_{\varphi^*}$, one sees that 
	$$ \frac{1}{J(e^{-\varphi})}= \frac{\int_{\Omega} (\varphi^*(y))^{p-1}\,d\nu(y)}  {\int_{\Omega} \,d \mu_{1}(e^{-\varphi} ,y) } = \frac{\int_{\Omega} \,d\nu(y)}  {\int_{\Omega} \,d \mu_{p}(e^{-\varphi}, y)}.$$
	This completes the proof of Theorem \ref{solution-lp-mink-6-2}.  \end{proof} 

\vskip 2mm \noindent  {\bf Acknowledgement.}   The
research of DY was supported by a NSERC grant, Canada.

\vskip 2mm \noindent Niufa Fang, \ \ \ {\small \tt fangniufa@nankai.edu.cn}\\
{\em Chern Institute of Mathematics, Nankai University, Tianjin, 300 071, China
}

\vskip 2mm \noindent Sudan Xing, \ \ \ {\small \tt sxing@ualberta.ca}\\
{\em Department of Mathematical and Statistical Sciences, University of Alberta, Edmonton,  Alberta T6G 2G1, Canada
} 

\vskip 2mm \noindent Deping Ye, \ \ \ {\small \tt deping.ye@mun.ca}\\
{\em Department of Mathematics and Statistics, Memorial
	University of Newfoundland, St. John's, Newfoundland A1C 5S7,
	Canada
} 
\end{document}